\documentclass[aos, preprint]{imsart}

\usepackage{graphicx}
\usepackage{caption}
\usepackage{subcaption}
\usepackage{algorithm}
\usepackage[noend]{algpseudocode}
\RequirePackage[OT1]{fontenc}
\RequirePackage{amsthm,amsmath,amsfonts}
\RequirePackage{natbib}
\RequirePackage[colorlinks,citecolor=blue,urlcolor=blue]{hyperref}

\startlocaldefs
\numberwithin{equation}{section}
\theoremstyle{plain}
\newtheorem{theorem}{Theorem}[section]
\endlocaldefs

\newtheorem{assumption}[theorem]{Assumption}
\newtheorem{corollary}[theorem]{Corollary}

\newtheorem{lemma}[theorem]{Lemma}

\newtheorem{remark}[theorem]{Remark}

\def\bbeta{\bar{\beta}}
\def\hbeta{\hat{\beta}}
\def\hmu{\hat{\mu}}

\def\hM{\hat{M}}

\def\cM{\mathcal{M}}
\def\cardM{|\cM|}
\def\Err{\textnormal{Err}}
\def\err{\textnormal{err}}

\def\mbb#1{\mathbb{#1}} 

\def\mrm#1{\mathrm{#1}}

\def\real{\mathbb{R}} 
\def\integers{\mathbb{Z}} 
\def\norm#1{\left\|{#1}\right\|} 
\def\indic#1{\mbb{I}\left({#1}\right)} 
\def\abv#1{\left|{#1}\right|}

\providecommand{\diag}{\mathop\mathrm{diag}}
\providecommand{\tr}{\mathop\mathrm{tr}}

\def\rank#1{\mathrm{rank}({#1})}

\def\E{\mathbb{E}} 
\def\Earg#1{\E\left[{#1}\right]}
\def\Esubarg#1#2{\E_{#1}\left[{#2}\right]}

\def\Cov{\mrm{Cov}} 
\def\Covarg#1{\Cov\left[{#1}\right]}

\def\Var{\mrm{Var}} 
\def\Vararg#1{\Var\left[{#1}\right]}

\def\df{\mrm{df}}

\begin{document}

\begin{frontmatter}
\title{Prediction error after model search}
\runtitle{Prediction error after model search}

\begin{aug}
\author{\fnms{Xiaoying Tian} \snm{Harris}\ead[label=e1]{xtian@stanford.edu}},

\runauthor{Xiaoying Tian Harris}

\affiliation{Department of Statistics, Stanford University}

\address{390 Serra Mall, Stanford, California, USA\\
\printead{e1}}

\end{aug}

\begin{abstract}
Estimation of the prediction error of a linear estimation rule is difficult if
the data analyst also 
use data to select a set of variables and construct the estimation rule using
only the selected variables. In this work, we
propose an asymptotically unbiased estimator for the prediction error after
model search. Under some additional mild assumptions, we show
that our estimator converges to the true prediction error in $L^2$ at the
rate of $O(n^{-1/2})$, with $n$ being the number of data points.
Our estimator applies to general selection procedures,
not requiring analytical forms for the selection. The number of variables
to select from can grow as an exponential factor of $n$, allowing applications
in high-dimensional data. It also allows model misspecifications, not requiring
linear underlying models. One application of our method is that it provides
an estimator for the degrees of freedom for many discontinuous estimation rules
like best subset selection or relaxed Lasso. 
Connection to Stein's Unbiased Risk Estimator is discussed.
We consider in-sample prediction errors in this work, with some extension
to out-of-sample errors in low dimensional, linear models. Examples such
as best subset selection and relaxed Lasso are considered in simulations,
where our estimator outperforms both $C_p$ and cross validation in various
settings.
\end{abstract}



\end{frontmatter}

\section{Introduction}
\label{sec:introduction}

In this paper, we consider a homoscedastic model with Gaussian errors. In particular,
\begin{equation}
\label{eq:model}
y = \mu(X) + \epsilon, \qquad \epsilon \sim N(0, \sigma^2 I),
\end{equation}
where the feature matrix $X \in \real^{n \times p}$ is considered
fixed, $y \in \real^n$ is the response, and the noise level $\sigma^2$
is considered known and fixed. Note the mean function $\mu: \real^{n \times p} \rightarrow \real^n$
is not necessarily linear in $X$.

Prediction problems involve finding an estimator $\hat{\mu}$ which fits the data well.
We naturally are interested in its performance in predicting a future response vector that is
generated from the same mechanism as $y$. \cite{mallows_cp} provided an unbiased estimator
for the prediction error when the estimator is linear
$$
\hat{\mu} = H y,
$$ 
where $H$ is an $n \times n$ matrix independent of the data $y$. $H$ is often referred to as
the hat matrix. But in recent context,
it is more and more unrealistic that the data analyst will not use the data to build a
linear estimation rule. $H$, in other words, depends on $y$. In this case, is there still hope to get
an unbiased estimator for the prediction error? In this article, we seek to address this
problem. 

Some examples that our theory will apply to are the following.
In the context
of model selection, the data analyst might use some techniques to select a subset of the
predictors $M$ to build the linear estimation rules. Such techniques can include the more principled
methods like LASSO \citep{tibshirani:lasso}, best subset selection,
forward stepwise regression and Least Angle Regression \citep{lars} or some heuristics
or even the combination of both. After the selection step, we simply project the data onto
the column space of $X_M$, the submatrix of $X$ that consists of $M$ columns,
and use that as our estimation rule. Specifically,
\begin{equation}
\label{eq:relaxed_lasso}
\begin{aligned}
&\hmu(y; X) = H_M \cdot y, \qquad H_M = P_M = X_M(X_M^T X_M)^{-1}X_M^T, \\ 
&M = \hM(y)
\end{aligned}
\end{equation}
where $\hM$ can be any selection rule and $P_M$ is the projection matrix onto
the column space of $X_M$. In the case when $M$ is selected by the LASSO, 
$\hmu = X_M\bbeta_M(y)$, and $\bbeta_M(y)$ is known as the {\em relaxed LASSO} solution
\citep{relaxed_lasso}.

In general, the range of $\hM$ does not have to be $2^{\{1,\dots,p\}}$, the
collection of all subsets of $\{1,2,\dots,p\}$, but we assume the hat matrix $H_{\hat{M}}$ depends on
the data $y$ only through $\hM$. In this sense, $\hM$ is the abstraction of the data-driven
part in $H$. This paper will study the prediction error of
$$
\hmu = H_{\hM} \cdot y.
$$
In this paper, we want to estimate the prediction error for $\hmu$,
\begin{equation}
\label{eq:pe:marginal}
\Err = \Earg{\|y_{new} - H_{\hM} \cdot y\|^2_2}, \quad y_{new} \sim N(\mu(X), \sigma^2 I) \perp y.
\end{equation}

There are several major methods for estimating \eqref{eq:pe:marginal} \citep{efron_cp}.
\begin{description}  
\item{\bf Penalty methods} such as $C_p$ or Akaike's information criterion (AIC) add a penalty term to
the loss in training data. The penalty is usually twice the degrees of freedom times $\sigma^2$.
\item{\bf Stein's Unbiased Risk Estimator} \citep{sure} provides an unbiased estimator for any estimator that is
smooth in the data. For non-smooth estimation rules, \cite{ye_perturbance} use perturbation techniques to approximate the
covariance term for general estimators.
\item{\bf Nonparametric methods} like cross validation or related bootstrap techniques provide risk estimators
without any model assumption.
\end{description} 

Methods like $C_p$ assume a fixed model. Or specifically, the degrees of freedom
is defined as $\df = \tr(H)$ for fixed $H$. Stein's Unbiased
Risk Estimator (SURE) only allows risk estimation for almost differentiable estimators. 
In addition, computing the SURE estimate usually involves calculating the divergence of
$\hmu(y)$. This is difficult when $\hmu(y)$ does not have an explicit form. Some special
cases have been considered. Works by \cite{zou2007degrees, tibshirani2012degrees} have
computed the ``degrees of freedom'' for the LASSO estimator, which
is Lipschitz. But for general estimators of the form $\hmu = H_{\hM} y$, where $H_{\hM}$
depends on $y$, $\hmu$ might not even be continuous in the data $y$. Thus analytical forms
of the prediction error are very difficult to derive \cite{tibshirani2014degrees, mikkelsen2016degrees}.

Nonparametric methods like cross validation are probably the most ubiquitous in practice. Cross validation
has the advantage of assuming almost no model assumptions. However, \cite{cv_high_dim} shows that cross
validation is inconsistent for estimating prediction error in high dimensional scenarios.  
Moreover, cross validation also includes extra variation from having a different $X$ for the validation set,
which is different from the fixed $X$ setup of this work.
\cite{efron_cp} also points out that the model-based methods like $C_p$, AIC, SURE offer substantially better
accuracy compared with cross validation, given the model is believable.

In this work, we introduce a method for estimating prediction errors that is applicable
to general model selection procedures. Examples include best subset selection for which
prediction errors are difficult to estimate beyond $X$ being orthogonal matrices \citep{tibshirani2014degrees}.
In general, we do not require $H_{\hM}$ to
have any analytical forms. The main approach is to apply the selection algorithm $\hM$ to
a slightly randomized response vector $y^*$. This is similar to holding out the validation
set in cross validation, with the distinction that we do not have to split the feature matrix
$X$. We can then construct an unbiased estimator for the prediction error using the holdout
information that is analogous to the validation set in cross validation. Note that since
$y^*$ would select a different model from $y$, this estimator will not be unbiased for the
prediction error of $\hmu$. However, If the perturbation
in $y^*$ is small and we repeat this process multiple times so the randomization averages out,
we will get an asymptotically unbiased and consistent estimator
for the prediction error of $\hmu = H_{\hM(y)} y$, which is the original target of our estimation.
Moreover, since our estimator is model based, it also enjoys more efficiency than cross validation.

In fact, we prove that under mild conditions on the
selection procedure, our estimator converges to the true prediction error as in \eqref{eq:pe:marginal}
in $L^2$ at the rate of $n^{-\frac{1}{2}}$. This automatically implies consistency of
our estimator. The $C_p$ estimator, on the other hand, converges in $L^2$
at the rate of $n^{-1}$ for fixed hat matrix $H$. So compared with $C_p$, our estimator
pays a small price of $n^{\frac{1}{2}}$ for the protection
against any ``data-driven'' manipulations in choosing the hat matrix $H_{\hM}$ for the linear estimation rules.

\subsection{Organization}

The rest of the paper is organized as follows. In Section \ref{sec:setup}, we
introduce our procedure for unbiased estimation for a slightly different
target.  This is achieved by first randomizing the data and then constructing
an unbiased estimator for the prediction error of this slightly different
estimation rule.  We then address the question of how randomization affects the
accuracy of our estimator for estimating the true prediction error.
There is a clear bias-variance trade-off with
respect to the amount of randomization. We derive upper bounds for the bias and the
variance in Section \ref{sec:bias:var} and propose an ``optimal'' scale of
randomization that would make our estimator converge to the true prediction
error in $L^2$.  Since the unbiased estimator constructed in Section
\ref{sec:setup} only uses one instance of randomization. We can further reduce
the variance of our estimator by averaging over different randomizations. In
Section \ref{sec:application}, we propose a simple algorithm to compute the
estimator after averaging over different randomizations.
We also discuss the condition under which our
estimator is equal to the SURE estimator. While SURE is difficult to compute
both in terms of analytical formula and simulation, our estimator is easy to
compute. Using the relationship between prediction error and degrees of freedom,
we also discuss how to compute the ``search degrees of freedom'', a term used in
\cite{tibshirani2014degrees} to refer to the degrees of freedom of estimators after model search. 
Finally, we include some simulation results in Section
\ref{sec:simulation} and conclude with some discussions in Section
\ref{sec:discussion}. 

\section{Method of estimation}
\label{sec:setup}

First, we assume the homoscedastic Gaussian model in \eqref{eq:model}, $y \sim N(\mu(X), \sigma^2 I)$,
and we have a {\em model selection} algorithm $\hM$,
$$
\hM:~\real^n \times \real^{n \times p} \rightarrow \cM, \quad (y, X) \mapsto M. 
$$
As we assume $X$ is fixed, we often use the shorthand $\hM(y)$, and assume
$$
\hM:~\real^n \rightarrow \cM, \quad y \mapsto M, 
$$
where $\cM$ is a {\bf finite} collection
of models we are potentially interested in. The definition of models here is quite general.
It can refer to any information we extract from the data. A common model as described in the
introduction can be a subset of predictors of particular interest. In such case, 
$\hM$ takes a value of the observation $y$ and maps it to a set of selected variables.
Note also the inverse image of $\hM^{-1}$ induces a partition on the space of $\real^n$.
We will discuss this partition further in Section \ref{sec:bias:var}. 

However, instead of using the original response variable $y$ for selection, we use its randomized
version $y^*$,
\begin{equation}
\label{eq:random:additive}
y^* = y + \omega, \qquad \omega \sim N(0, \alpha \sigma^2 I) ~ \perp y.
\end{equation}

For a fixed $\alpha > 0$, after using $y^*$ to select a model $M$, we can define
prediction errors analogous to that defined in \eqref{eq:pe:marginal},
\begin{equation}
\label{eq:pe:R:marginal}
\Err_{\alpha} = \Earg{\|y_{new} - H_{\hM(y+\omega)} y\|^2_2 }, \qquad  
y_{new} \sim N(\mu(X), \sigma^2 I) ~\perp (y,\omega).
\end{equation}
The subscript $\alpha$ denotes the amount of randomization added to $y$. Note that although
randomization noise $\omega$ is added to selection, $\Err_{\alpha}$
integrates over such randomization and thus are {\bf not} random.
The prediction error $\Err$ as defined in \eqref{eq:pe:marginal}
corresponds to the case where we set $\alpha = 0$. In this section, we show that
can get an unbiased estimator for $\Err_{\alpha}$ for any $\alpha > 0$.
Before we introduce the unbiased estimator, we first introduce some background on randomization.
 
\subsection{Randomized selection}

It might seem unusual to use $y^*$ for model selection. But actually, using randomization
for model selection and fitting is quite common -- the 
common practice of splitting the data into a training set and a test set is a form of
randomization. Although not stressed, the split is usually random and thus we are using
a random subset of the data instead of the data itself for model selection and training.

The idea of randomization for model selection is not new. The field of differential privacy
uses randomized data for database queries to preserve information \cite{dwork2008differential}. 
This particular additive randomization scheme, $y^* = y+\omega$ is discussed in \cite{randomized_response}.
In this work, we discover that the additive randomization in \eqref{eq:random:additive} allows
us to construct a vector independent of the model selection. This independent vector is analogous to the
validation set in data splitting.

To address the question of the effect of randomization, we prove that $\Err$ and 
$\Err_{\alpha}$ are close for small $\alpha > 0$ under mild conditions on the selection procedures. 
In other words, since we have an unbiased estimator for $\Err_{\alpha}$ for any $\alpha > 0$, when
$\alpha$ goes to $0$, its bias for $\Err$ will diminish as well.  
For details, see Section \ref{sec:bias:var}.
In addition, Section \ref{sec:simulation} also provides some evidence in simulations.

\subsection{Unbiased estimation}
\label{sec:unbiased}

To construct an unbiased estimator for $\Err_{\alpha}$, we first construct the following vector
that is independent of $y^*$,
\begin{equation}
y^- = y - \frac{1}{\alpha} \omega.
\end{equation}
Note this construction is also mentioned in \cite{randomized_response}. Using the property of Gaussian
distributions and calculating the covariance between $y^-$ and $y^*=y+\omega$, it is easy to see
$y^-$ is independent of $y^*$, and thus independent of the selection event $\{\hM(y^*) = M\}$. 
Now we state our first result that constructs an unbiased estimator for $\Err_{\alpha}$ for
any $\alpha > 0$.

\begin{theorem}[Unbiased Estimator]
\label{thm:unbiased}
Suppose $y \sim N(\mu(X), \sigma^2 I)$ is from the homoscedastic Gaussian model \eqref{eq:model},
then
\begin{equation}
\label{eq:estimate:marginal}
\widehat{\Err}_{\alpha} = \norm{y^- - H_{\hM(y^*)} y}_2^2 + 2\tr(H_{\hM(y^*)}) \sigma^2 - \frac{1}{\alpha} n \sigma^2 
\end{equation}
is unbiased for $\Err_{\alpha}$ for any $\alpha > 0$.
\end{theorem} 

\begin{proof}
First notice
$$
y = \frac{1}{1+\alpha} y^* + \frac{\alpha}{1+\alpha} y^-, \quad y = \mu(X) + \epsilon,
$$
if we let $\epsilon^* = y^* - \mu(X)$ and $\epsilon^- = y^- - \mu(X)$, then
\begin{equation}
\label{eq:decomp}
\epsilon = \frac{1}{1+\alpha} \epsilon^* +  \frac{\alpha}{1+\alpha} \epsilon^-.
\end{equation}
Note $\epsilon^* \perp \epsilon^-$ and $\epsilon^* \sim N(0, (1+\alpha)\sigma^2 I)$,
and $\epsilon^- \sim N(0, \frac{1+\alpha}{\alpha} \sigma^2 I)$.

With this, we first define the following estimator for any $\alpha>0$ and any $M \in \cM$,
\begin{equation}
\label{eq:estimate:conditional}
\widehat{\err}_{\alpha}(M) = \norm{y^- - H_M y}_2^2 + 2\tr(H_M) \sigma^2 - \frac{1}{\alpha} n \sigma^2. 
\end{equation}

We claim that $\widehat{\err}_{\alpha}(M)$ is unbiased for the prediction error conditional on $\{\hM(y^*) = M\}$
for any $M \in \cM$ and any $\alpha > 0$. Formally, we prove
\begin{equation}
\label{eq:conditional:equality}
\Earg{\widehat{\err}_{\alpha}(M) \mid \hM(y^*) = M} = \Earg{\norm{y_{new} - H_{M}\cdot y}^2 \mid \hM(y^*)=M}.
\end{equation}

To see \eqref{eq:conditional:equality}, we first rewrite
\begin{equation} 
\label{eq:error:rr}
\Earg{\norm{y_{new} - H_M y}_2^2 \mid \hM(y^*) = M}
= \Earg{\norm{\mu-H_M y}^2 \mid \hM(y^*) = M} + n \sigma^2.
\end{equation} 

Now we consider the conditional expectation of $\widehat{\err}_{\alpha}(M)$. Note
$$
\begin{aligned}
&\Earg{\norm{y^- - H_M y}_2^2 \mid \hM(y^*) = M} \\
= &\Earg{\norm{\mu-H_M  y}^2 \mid \hM(y^*) = M} + \frac{1+\alpha}{\alpha}n\sigma^2
- 2\Earg{(\epsilon^-)^T H_M y \mid \hM(y^*) = M} \\
=& \Earg{\norm{\mu-H_M  y}^2 \mid \hM(y^*) = M} + \frac{1+\alpha}{\alpha}n\sigma^2\
- \frac{2\alpha}{1+\alpha}\tr\left[H_M \Earg{y^- (\epsilon^-)^T}\right]\\ 
=& \Earg{\norm{\mu-H_M  y}^2 \mid \hM(y^*) = M} + \frac{1+\alpha}{\alpha}n\sigma^2 
- 2\tr(H_M)\sigma^2 
\end{aligned}
$$
The equalities use the decomposition \eqref{eq:decomp} as well as the
fact that $y^* \perp \epsilon^-$. 

Comparing this with \eqref{eq:error:rr}, it is easy to see \eqref{eq:conditional:equality}.
Moreover, marginalizing over $\hM(y^*)$, it is easy to see $\widehat{\Err}_{\alpha}$ in
\eqref{eq:estimate:marginal} is unbiased for $\Err_{\alpha}$.
\end{proof}

In fact, using the proof for Theorem \ref{thm:unbiased}, we have a even stronger result
than the unbiasedness of $\widehat{\Err}_{\alpha}$.

\begin{remark}
$\widehat{\Err}_{\alpha}$ is not only unbiased for the prediction error marginally, but
conditional on any selected event $\{\hM(y^*) = M\}$, $\widehat{\Err}_{\alpha}$ is also
unbiased for the prediction error. Formally,
$$
\Earg{\widehat{\Err}_{\alpha} \mid \hM(y^*) = M} 
= \Earg{\norm{y_{new} - H_{M}\cdot y}^2 \mid \hM(y^*)=M}.
$$
This is easy to see with \eqref{eq:conditional:equality} and
$$
\Earg{\widehat{\Err}_{\alpha} \mid \hM(y^*) = M} = \Earg{\widehat{\err}_{\alpha}(M) \mid \hM(y^*) = M}.
$$
\end{remark}

The simple form of $\widehat{\Err}_{\alpha}$ in \eqref{eq:estimate:conditional}
has some resemblance to the usual $C_p$ formula for prediction error estimation,
with $2\tr(H_{\hM})\sigma^2$ being the usual correction term for degrees of freedom in
the $C_p$ estimator. The additional term $n\sigma^2/\alpha$ helps offset the larger
variance in $y^-$.

\section{Randomization and the bias-variance trade-off}
\label{sec:bias:var}

We investigate the effect of randomization in this section. In particular, we are interested in
the difference $\Err_{\alpha} - \Err$ for small $\alpha > 0$ and the variance of our estimator
$\widehat{\Err}_{\alpha}$ as a function of $\alpha$. 

There is a simple intuition for the effects of randomization on estimation of prediction error.
Since $y^* = y + \omega, ~ \omega \sim N(0, \alpha \sigma^2)$, the randomized response vector $y^*$
which we use for model selection will be close to $y$ when the randomization scale $\alpha$ is small.
Intuitively, $\Err_{\alpha}$ should be closer to $\Err$ when $\alpha$ decreases. On the other hand,
the independent vector $y^- = y - \omega / \alpha$ which we use to construct the estimator for the
prediction error is more variant when $\alpha$ is small. Thus, there is a clear bias-variance trade-off
in the choice of $\alpha$. We seek to find the optimal scale of $\alpha$ for this trade-off.

The unbiased estimation introduced in Section \ref{sec:unbiased} does not place any assumptions on
the hat matrix $H_{\hM}$ or the selection procedure $\hM$. However, in this section we restrict
the hat matrices to be those constructed with the selected columns of the design matrix $X$. 
This is not much of a restriction, since these hat matrices are probably of most interest.
Formally,
we restrict $\hM$ to be a selection procedure that selects an ``important'' subset of variables. That is
$$
\hM: \real^n \to \cM \subseteq 2^{\{1,\dots,p\}}.
$$ 
Without loss of generality, we assume $\hM$ is surjective. Thus the number of potential models to choose
from is $\cardM$ which is finite. Moreover, the map $\hM$ induces a partition of the space of $\real^n$.
In particular, we assume
\begin{equation}
\label{eq:partition}
U_i = \hM^{-1}(M_i) \subseteq \real^n, \quad i = 1, \dots, \cardM, 
\end{equation}
where $\cM = \{M_1, \dots, M_{\cardM}\}$ are different models to choose from.
It is easy to see that
$$
\coprod_{i=1}^{\cardM} U_i = \real^n,
$$
and we further assume $\textrm{int} (U_i) \neq \emptyset$ and $\partial U_i$
has measure 0 under the Lebesgue measure on $\real^n$.  

Now we assume the hat matrix is a constant matrix in each of the partition $U_i$. In particular, 
\begin{equation}
\label{eq:hat:matrix}
H_{\hM(y)} = \sum_{i=1}^{\cardM} H_{M_i} \indic{y \in U_i}.
\end{equation}

The most common matrix is probably the projection matrix onto the column space spanned by a subset
of variables. Formally, we assume
\begin{assumption} 
\label{A:nonexpansive}
For any $M \in \cM$, we assume $H_M = X_M (X_M^T X_M)^{-1} X_M^T$, where $X_M$ is the submatrix of $X$
with $M$ as the selected columns. It is easy to see, $H_M$ is symmetric and 
$$
H_M^2 = H_M, \quad \forall ~M \in \cM.
$$
\end{assumption} 

Moreover, we also assume that $\hM$ does not select too many variables to include in a model,
and the rank of the matrix $\rank{H_M}$ is less than number of variables in $M$. Specifically,
\begin{assumption}
\label{A:trace:bound}
For any $M \in \cM$, 
$$
\rank{H_M} \leq  |M| \leq K.
$$
Furthermore, we assume $K$ grows with $n$ at the following rate,
\begin{equation}
\label{eq:sizes}
K^2 \log p = O(\sqrt{n}),
\end{equation}
where $p$ is the number of columns in $X$.
\end{assumption}

Assumption \ref{A:trace:bound} requires that
none of the models $M \in \cM$ is too large. However, its size can grow with $n$ at some
polynomial rate. For
penalized regression problems, choices of the penalty parameter $\lambda$ such that the
solution is sparse has been studied in \cite{negahban2009unified}. 
Also Assumption \ref{A:trace:bound} allows $p$ to grow as an exponential factor of the number of data points $n$. 
Thus it allows applications in the high dimensional setting where $p \gg n$.

We also assume the model selection procedure $\hM$ to have reasonable accuracy in identifying the underlying model. 
\begin{assumption}
\label{A:accuracy}
Suppose $\hM$ satisfies:
\begin{equation}
\label{eq:accuracy}
\Earg{\left(\norm{\mu - H_{\hM(y')}\cdot\mu}_2^2\right)^2} = O(n), \quad y'\sim N(\mu, \tau),
\end{equation}
where $\sigma^2 \leq \tau \leq (1+\delta)\sigma^2$ for some small constant $\delta>0$.
\end{assumption}

Assumption \ref{A:accuracy} assumes that the subspace spanned by $X_{\hM}$ is a
good representation of the underlying mean $\mu$. But it does not require $\mu$
to be in this subspace. Namely, we allow model misspecifications. 

\begin{remark}
In the context of sparse
linear models $\mu=X\beta^0$, we assume $\beta^0$ has support $\Gamma$. Then
$$
\begin{aligned}
&\norm{\mu - H_{\hM}\cdot \mu}^2 \\
=& \norm{(I - H_{\hM}) X_{\Gamma \backslash \hM} \beta^0_{\Gamma\backslash \hM}}^2\\
\leq& \norm{X_{\Gamma \backslash \hM} \beta^0_{\Gamma\backslash\hM}}^2
\end{aligned}
$$ 
Assuming $X$ is normalized to have column length $\sqrt{n}$ and that $\beta^0_j = O\left(\sqrt{\frac{\log p}{n}}\right)$
for any $j \in \Gamma$, we have that
$$
\norm{\mu - H_{\hM} \cdot \mu}_2^2 = \abv{\Gamma \backslash \hM}^2 \log p.
$$
Thus Assumption \ref{A:accuracy} is close to placing a condition like 
$\abv{\Gamma \backslash \hM}^2 \log p = O_p(\sqrt{n})$, which is analogous
to that in Equation \eqref{eq:sizes}. 
\end{remark}

With these conditions above, we show in the following that the bias of $\widehat{\Err}_{\alpha}$ is $O(\alpha)$
(Theorem \ref{thm:bias}) and its variance is $O((n\alpha^2)^{-1})$ (Theorem \ref{thm:var}).
There is a clear bias variance trade-off with regard to the choice of $\alpha$, which we will discuss
in more detail in Section \ref{sec:tradeoff}. The proofs of the theorems uses some well known results
in extreme value theories.

\subsection{Bias}
\label{sec:bias}

The bias of $\widehat{\Err}_{\alpha}$ is introduced by the fact that selection is performed with
$y^*$, the randomized version of $y$. In fact, this is the only source for the bias. However,
for small perturbations, the resulting bias will be small as well. Formally, we have the following
theorem.
\begin{theorem}
\label{thm:bias}
Suppose Assumptions \ref{A:nonexpansive}, \ref{A:trace:bound} and \ref{A:accuracy} are satisfied, then the bias
of $\widehat{\Err}_{\alpha}$ is bounded by:
$$
\frac{1}{n}\abv{\Earg{\widehat{\Err}_{\alpha}} - \Err} = \frac{1}{n}\abv{\Err_{\alpha} - \Err} \leq C \alpha,
\quad \text{ for } \alpha < \delta.
$$
where $C$ is a universal constant and $\delta > 0$ is a small constant defined in Assumption \ref{A:accuracy}.
\end{theorem}

Essential to the proof of Theorem \ref{thm:bias} is that the performance of the estimation rule
$$
\hmu(y) = H_{\hM(y)} y
$$
is resistant to small perturbations on $y$.
This is true under the assumptions introduced at the beginning of Section \ref{sec:bias:var}.
Formally, we have
\begin{lemma}
\label{lem:linear:alpha}
Suppose our hat matrix is of the form in \eqref{eq:hat:matrix} and Assumptions \ref{A:nonexpansive}-\ref{A:accuracy}
are satisfied. Then for $\alpha < \delta$ and $\omega \sim N(0, \alpha\sigma^2 I)$, we have
$$
\frac{1}{n}\abv{\Earg{\norm{\hmu(y + \omega)-\mu}^2} - \Earg{\norm{\hmu(y)-\mu}^2}} \leq C_1 \cdot \alpha,
$$
where $C_1$ is a universal constant and $\delta > 0$ is a small constant defined in Assumption \ref{A:accuracy}.
The first expectation is taken over $(y, \omega)$ and the second expectation is taken over $y$. 
\end{lemma}

With Lemma \ref{lem:linear:alpha}, it is easy to prove Theorem \ref{thm:bias}.
\begin{proof}
First notice that
$$
H_{\hM(y^*)} y = H_{\hM(y^*)} y^* - H_{\hM(y^*)} \omega = \hmu(y+\omega) - H_{\hM(y+\omega)} \omega.
$$
Thus we have
$$
\begin{aligned}
&\frac{1}{n}\abv{\Err_{\alpha} - \Err} \\
=& \frac{1}{n}\abv{\Earg{\norm{H_{\hM(y+\omega) }y - \mu}^2} - \Earg{\norm{H_{\hM(y) }y - \mu}^2}} \\
\leq& \frac{1}{n}\abv{\Earg{\norm{\hmu(y+\omega) - \mu}^2} - \Earg{\norm{\hmu(y) - \mu}^2}} + \frac{1}{n}\Earg{\norm{H_{\hM(y+\omega)}\omega}^2} \\
\leq& \frac{1}{n}\abv{\Earg{\norm{\hmu(y+\omega) - \mu}^2} - \Earg{\norm{\hmu(y) - \mu}^2}} + \frac{1}{n}\Earg{\norm{\omega}^2} \\
\leq& (C_1 + \sigma^2) \alpha
\end{aligned}
$$
\end{proof}

The proof of Lemma \ref{lem:linear:alpha} relies on the following lemma which will also be used for the bound
on the variance of $\widehat{\Err}_{\alpha}$ in Section \ref{sec:var}. The proofs of both Lemma \ref{lem:linear:alpha} and Lemma \ref{lem:max:rand}
are deferred to Section \ref{sec:proof}. 
\begin{lemma}
\label{lem:max:rand}
Suppose $Z \sim N(0, I_{n \times n})$, and $H_{\hM}$ is a hat matrix that
takes value in
$$
\left\{H_{M_i},~~i=1, \dots, \cardM \right\},
$$ 
where $\cardM$ is the total number of potential models to choose from. The value of
$H_{\hM}$ may depend on $Z$. If we further assume $H_{\hM}$ satisfies  
Assumptions \ref{A:nonexpansive} and \ref{A:trace:bound}, then
$$
\Earg{\left(\norm{H_{\hM} Z}_2^2\right)^2} \leq 18 \bigg[K \log (K\cardM) \bigg]^2.
$$
\end{lemma}

\subsection{Variance}
\label{sec:var}

In this section, we discuss the variance of our estimator $\widehat{\Err}_{\alpha}$.
As previously discussed at the beginning of the section,
it is intuitive that the variance of $\widehat{\Err}_{\alpha}$ will
increase as $\alpha$ decreases. Before we establish quantitative results about the variances of $\widehat{\Err}_{\alpha}$
with respect to $\alpha$, we first state a result on the variances of $C_p$ estimators. This will provide
a baseline of comparison for the increase of variance due to the model selection procedure. Formally,
suppose $y \sim N(\mu, \sigma^2 I)$ and the hat matrix $H_{\hM}$ is constant
independent of the data $y$, that is $H_{\hM} = H$. Then the $C_p$ estimator
$$
C_p = \norm{y - H y}_2^2 + \tr(H) \sigma^2  
$$
is unbiased for the prediction error. Moreover

\begin{lemma}
\footnote{
Lemma \ref{lem:Cp} is inspired by a talk given by Professor Lawrence Brown, although the author has not been
able to find formal proof for reference.}
\label{lem:Cp}
Suppose $y \sim N(\mu, \sigma^2 I)$ and $H_{\hM} = H$ is a constant projection matrix,
$$
\Vararg{C_p} =  2\tr\left[(I-H)\right]\sigma^4 + 4\norm{(I-H)\mu}_2^2\sigma^2.
$$
Furthermore, if we assume $\norm{\mu}_2^2 = O(n)$,
\begin{equation}
\label{eq:var:Cp}
\Vararg{\frac{C_p}{n}} = O\left(\frac{1}{n}\right)
\end{equation}
\end{lemma}

The proof of Lemma \ref{lem:Cp} uses the following lemma whose proof we defer to Section \ref{sec:proof}:
\begin{lemma}
\label{lem:norm:var}
Let $Z \sim N(0, I_{n \times n})$ and $A \in \real^{n \times n}$ be any fixed matrix,
$$
\Vararg{\norm{AZ}_2^2} = 2 \tr\left(A^4\right)
$$
\end{lemma}
Now we prove Lemma \ref{lem:Cp}.

\begin{proof}
$$
\begin{aligned}
\Vararg{\norm{y - H y}^2} &= \Vararg{\norm{(I-H) \epsilon}^2 + \norm{(I-H) \mu}^2 + 2\mu^T(I-H)^2\epsilon}\\
&= \Vararg{\norm{(I-H)\epsilon}^2 + 2\mu^T(I-H)^2\epsilon}\\
&= \Vararg{\norm{(I-H)\epsilon}^2} + 4\Vararg{\mu^T(I-H)^2\epsilon} \\
&= 2\tr\left[(I-H)^4\right]\sigma^4 + 4\norm{(I-H)^2\mu}^2 \sigma^2\\
&= 2\tr\left[(I-H)\right]\sigma^4 + 4\norm{(I-H)\mu}^2 \sigma^2
\end{aligned}
$$
The last equality is per Assumption \ref{A:nonexpansive}, as it is easy to see that $I-H$ is also a projection matrix. 
Finally, since
$$
\tr(I-H) \leq n, \quad \norm{(I-H) \mu}_2^2 \leq \norm{\mu}_2^2,
$$
it is easy to deduce the second conclusion. 
\end{proof}

Lemma \ref{lem:Cp} states that the variation in $\frac{C_p}{n}$ is of order $O_p(n^{-\frac{1}{2}})$. In the following,
we seek to establish how inflated the variance of $\widehat{\Err}_{\alpha}$ is compared to $C_p$.
Theorem \ref{thm:var} gives an explicit upper bound on the variance of $\widehat{\Err}_{\alpha}$
with respect to $\alpha$. In fact, it is simply of order $O((n\alpha^2)^{-1})$.
Formally, we have
\begin{theorem}
\label{thm:var}
Suppose Assumptions \ref{A:nonexpansive}, \ref{A:trace:bound} and \ref{A:accuracy} are satisfied, then
$$
\Vararg{\frac{\widehat{\Err}_{\alpha}}{n}} = O\left(\frac{1}{n\alpha^2}\right).
$$
\end{theorem}

Compared with the variance of the $C_p/n$ in \eqref{eq:var:Cp}, we pay a price of $\alpha^{-2}$ but
allows our hat matrix to be dependent on the data $y$. Particularly, if we choose $\alpha^{-1} = o(n^{1/2})$,
our estimator for the prediction error $\widehat{\Err}_{\alpha}$ will have diminishing variance and be consistent. 

To prove Theorem \ref{thm:var}, we again use Lemma \ref{lem:max:rand} which is stated in Section \ref{sec:bias}.
\begin{proof}
First notice that $y = y^- - \frac{1}{\alpha}\omega$, thus
$$
\norm{y^- - H_{\hM(y^*)} \cdot y}^2 = \norm{y^- - H_{\hM(y^*)} \cdot y^-}^2 + \frac{1}{\alpha^2} \norm{H_{\hM(y^*)} \omega}^2.
$$
Therefore, we have,
$$
\begin{aligned}
&\Vararg{\widehat{\Err}_{\alpha}}
= \Vararg{\norm{y^- - H_{\hM(y^*)} \cdot y}^2 + 2\tr(H_{\hM(y^*)})\sigma^2}\\
\leq& 2\Vararg{\norm{y^- - H_{\hM(y^*)} y^-}^2 + 2\tr\left[H_{\hM(y^*)}\right]\sigma^2}
+ 2 \Vararg{\frac{1}{\alpha^2} \norm{H_{\hM(y^*)}\omega}^2}
\end{aligned}
$$
First using the decomposition for conditional variance, we have 
$$
\begin{aligned}
\hspace{10pt}&\Vararg{\norm{y^- - H_{\hM(y^*)} y^-}^2 + 2\tr\left[H_{\hM(y^*)}\right]\sigma^2} \\
=& \Earg{\Vararg{\norm{y^- - H_{M} y^-}^2 + 2\tr\left(H_{M}\right)\sigma^2 \bigg| \hM(y^*) = M}}\\
 &+ \Vararg{\norm{(I - H_{\hM(y^*)}) \mu}_2^2}\\
\leq& 2n\left(1+\frac{1}{\alpha}\right)^2\sigma^4 + 4\Earg{\norm{(I-H_{\hM(y^*)})\mu}_2^2} \left(1+\frac{1}{\alpha}\right) \sigma^2 \\
&\hspace{10pt} + \Vararg{\norm{(I - H_{\hM(y^*)}) \mu}_2^2}
\end{aligned}
$$
Note for the last inequality, we used both Lemma \ref{lem:Cp} as well as the
independence relationships: $y^* \perp y^-$. Furthermore, with Assumption \ref{A:accuracy}, we have
$$
\Vararg{\norm{y^- - H_{\hM(y^*)} y^-}^2 + 2\tr\left[H_{\hM(y^*)}\right]\sigma^2} \leq O(n\alpha^{-2}).
$$
Moreover, per Lemma \ref{lem:max:rand} we have
$$
\Vararg{\frac{1}{\alpha^2} \norm{H_{\hM(y^*)}\omega}^2} \leq \frac{1}{\alpha^4}\Earg{\left[\norm{H_{\hM(y^*)}\omega}^2\right]^2} 
\leq \frac{18}{\alpha^2} \left[K\log(K\cardM)\right]^2 \sigma^4, 
$$ 
where $\cardM$ is the number of potential models in $\cM$. Finally, since we would only choose models
of size less than $K$ (Assumption \ref{A:trace:bound}), $\cardM \leq p^{K}$ and with
the rate specified in Assumption \ref{A:trace:bound} we have the conclusion of the theorem.
\end{proof}

\subsection{Bias-variance trade-off and the choice of $\alpha$}
\label{sec:tradeoff}

After establishing Theorem \ref{thm:bias} and Theorem \ref{thm:var}, we combine the results and
summarize the bias-variance trade-off in the following corollary.

\begin{corollary} 
\label{cor:l2}
Suppose Assumptions \ref{A:nonexpansive}, \ref{A:trace:bound} and \ref{A:accuracy}
are satisfied, then we have
$$
\begin{aligned}
&\mathrm{Bias}\left[\frac{1}{n}\widehat{\Err}_{\alpha}\right] = O(\alpha)\\
&\Vararg{\frac{1}{n}\widehat{\Err}_{\alpha}} = O\left(\frac{1}{n\alpha^2}\right)
\end{aligned}
$$
Furthermore, if we choose $\alpha = n^{-\frac{1}{4}}$,
$$
\Earg{\frac{1}{n^2}\left[\widehat{\Err}_{\alpha} - \Err \right]^2} = O\left(n^{-\frac{1}{2}}\right). 
$$ 
\end{corollary} 

\begin{proof}
The first part of the corollary is a straightforward combination of Theorem \ref{thm:bias} and Theorem
\ref{thm:var}. Moreover, if we choose $\alpha = n^{-\frac{1}{4}}$,
$$
\Earg{\frac{1}{n^2}\left[\widehat{\Err}_{\alpha} - \Err \right]^2} = O(\alpha^2) + O\left(\frac{1}{n\alpha^2}\right) = O\left(n^{-\frac{1}{2}}\right).
$$
\end{proof}

It is easy to see the optimal rate of $\alpha$ that strikes a balance between bias and variance is
exactly $\alpha = n^{-\frac{1}{4}}$. This should offer some guidance about the choice of $\alpha$ in practice.

\section{Further properties and applications}
\label{sec:application}

In Section \ref{sec:tradeoff}, we show that $\widehat{\Err}_{\alpha}$ will have diminishing variances
if $\alpha$ is chosen properly. However, since $\widehat{\Err}_{\alpha}$
is computed using only one instance of the randomization $\omega$, its variance
can be further reduced if we aggregate over different randomizations $\omega$. Furthermore,
in the following section, we will show that after such marginalization over $\omega$, 
$\widehat{\Err}_{\alpha}$ is Uniform Minimum Variance Unbiased (UMVU) estimators
for the prediction error $\Err_{\alpha} $under some conditions.

\subsection{Variance reduction techniques and UMVU estimators}

We first introduce the following lemma that shows the variance of $\widehat{\Err}_{\alpha}$
can be reduced at no further assumption.

\begin{lemma}
The following estimator is unbiased for $\Err_{\alpha}$,
\begin{equation}
\label{eq:marginalized}
\widehat{\Err}^{(I)}_{\alpha} = \Esubarg{\omega}{\widehat{\Err}_{\alpha} \mid y},
\end{equation}
Furthermore, it has smaller variance,
$$
\Vararg{\widehat{\Err}^{(I)}_{\alpha}} \leq \Vararg{\widehat{\Err}_{\alpha}},
$$
\end{lemma}

The lemma can be easily proved using basic properties of conditional expectation. In practice,
we approximate the integration over $\omega$ by repeatedly sampling $\omega$ and taking the
averages. Specifically, Algorithm \ref{alg:marginalized} provides an algorithm for computing
$\widehat{\Err}_{\alpha}^{(I)}$ for any $\alpha > 0$. 

\begin{algorithm}
\caption{Algorithm for computing $\widehat{\Err}_{\alpha}^{(I)}$ for any $\alpha > 0$.}\label{alg:marginalized}
\begin{algorithmic}[1]
\State \textbf{Input:} $X$, $y$
\State \textbf{Initialize:} $\widehat{\Err}_{\alpha}^{(I)} \gets 0$, $N \in \integers_+$ 
\For{$i$ \texttt{ in } $1:N$  }
\State Draw $\omega^{(i)} \sim N(0, \alpha \sigma^2 I)$
\State Compute $y^* = y + \omega^{(i)}$, $y^- = y - \frac{1}{\alpha}\omega^{(i)}$
\State Compute $\hM^* = \hM(y^*)$
\State Use Equation \eqref{eq:estimate:marginal} to compute $\widehat{\Err}_{\alpha}^{(i)}$ from $y, y^-, \hM^*$.
\State $\widehat{\Err}_{\alpha}^{(I)} += \widehat{\Err}_{\alpha}^{(i)} / N$ 
\EndFor
\Return $\widehat{\Err}_{\alpha}^{(I)}$
\end{algorithmic}
\end{algorithm}
 
Since $\widehat{\Err}_{\alpha}^{(I)}$ has the same expectation as $\widehat{\Err}_{\alpha}$ with smaller
variances, it is easy to deduce from Corollary \ref{cor:l2} that $\widehat{\Err}_{\alpha}^{(I)}$
also converges to $\Err$ in $L^2$ at a rate of at least $O(n^{-\frac{1}{2}})$ (after a proper scaling of $n^{-1}$).
Furthermore, we show that such estimators are UMVU estimators for any $\alpha > 0$ when the parameter space
$\mu(X)$ contains a ball in $\real^n$. 

\begin{lemma}
\label{lem:umvu}
If parameter space of $\mu(X)$ contains a ball in $\real^n$, then $\widehat{\Err}_{\alpha}^{(I)}$
are UMVU estimators for $\Err_{\alpha}$ for any $\alpha > 0$.
\end{lemma}

\begin{proof}

Without loss of generality, assume $\omega$ has density $g$ with respect to the
Lebesgue measure on $\real^n$, then the density of $(y, \omega)$ 
with respect to the Lebesgue measure on $\real^n \times \real^n$ is proportional to
\begin{equation} 
\label{eq:law:marginal}
\exp\left[\frac{\mu(X)^Ty}{\sigma^2}\right] g(\omega). 
\end{equation} 

However, since \eqref{eq:law:marginal} is an exponential family with sufficient statistics $y$. Moreover, when
the parameter space of $\mu(X)$ contains a ball in $\real^n$, 
then we have $y$ is sufficient and complete. Thus taking an unbiased estimator $\widehat{\Err}_{\alpha}$
and integrating over $\omega$ conditional on $y$, the complete and sufficient statistics, we have the UMVU estimators. 
\end{proof}

\subsection{Relation to the SURE estimator}

In this section, we reveal that our estimator $\widehat{\Err}_{\alpha}$ is equal to
the SURE estimator for the prediction error $\Err_{\alpha}$ if the parameter space of $\mu(X)$ contains
a ball in $\real^n$.

First, we notice that for any $\alpha > 0$, $\Err_{\alpha}$ is the prediction error for
$$
\hmu_{\alpha}(y) = \Earg{H_{\hM(y+\omega)} y \mid y}, \quad \omega \sim N(0, \alpha \sigma^2 I).
$$
Although $\hmu(y)$ might be discontinuous in $y$, $\hmu_{\alpha}(y)$ is actually smooth
in the data. To see that, note
\begin{equation}
\label{eq:sure:integral}
\hmu_{\alpha}(y) = \sum_{i=1}^{\cardM} H_i y \int_{U_i} \phi_{\alpha}(z + y) dz,
\end{equation}
where $\phi_{\alpha}$ is the p.d.f for $N(0, \alpha\sigma^2 I)$. Due to the smoothness of
$\phi_{\alpha}$ and the summation being a finite sum, we have $\hmu_{\alpha}(y)$ is smooth
in $y$. Therefore, in theory we can use Stein's formula to compute an estimate for
the prediction error of $\hmu_{\alpha}(y)$. Note such estimator would only depend on $y$,
the complete and sufficient statistics for the exponential family in \eqref{eq:law:marginal}
when the parameter space of $\mu(X)$ contains a ball in $\real^n$.
Thus it is also the UMVU estimator for $\Err_{\alpha}$. By Lemma \ref{lem:umvu} and
the uniqueness of UMVU estimators, we conclude $\widehat{\Err}_{\alpha}$ is the same as
the SURE estimator.

However, the SURE estimator is quite difficult to compute as
the regions $U_i$'s may have complex geometry and explicit formulas are hard to derive \citep{mikkelsen2016degrees}.
Moreover, it is difficult to even use Monte-Carlo samplers to approximate the integrals in \eqref{eq:sure:integral} since
the sets $U_i$'s might be hard to describe and there are $\cardM$ integrals to evaluate, making
it computationally expensive. 

In contrast, $\widehat{\Err}_{\alpha}$ provides an unbiased estimator for $\Err_{\alpha}$ at a
much lower computational cost. That is we only need to sample $\omega$'s from $N(0, \alpha\sigma^2 I)$
and compute $\widehat{\Err}_{\alpha}$ at each time and average over them. The major computation
involved is re-selecting the model with $y^* = y+\omega$. In practice, we choose the number of samples
for $\omega$'s  to be less
than the number of data points, so the computation involved will be even less than Leave-One-Out
cross validation.

\subsection{Prediction error after model selection}
\label{sec:out_of_sample}

One key message of this work is that we can estimate the prediction error of the estimation rule $\hmu$
even if we have used some model selection procedure to construct the hat matrix $H_{\hM}$ in $\hmu$.
In practice, however, we need a priori information on $\sigma^2$ to compute $\widehat{\Err}_{\alpha}$.
There are several methods for consistent estimation of $\sigma^2$.
In the low dimensional setting, we can simply use the residual sum of squares divided by the degrees
of freedom to estimate $\sigma^2$. In the high dimensional setting, the problem is more challenging,
but various methods are derived including \cite{reid2013study, scaled_lasso, sqroot_lasso}.

We also want to stress that the prediction error defined in this work is the
in-sample prediction error that assumes fixed $X$. This is the same setup as in
$C_p$ \citep{mallows_cp}, SURE \citep{sure} and the prediction errors discussed
in \cite{efron_cp}.  A good estimator for the in-sample prediction error will
allow us to evaluate and compare the predictive power of different estimation
rules.

However, in other cases, we might be interested in out-of-sample prediction errors. That is, the 
prediction errors are measured on a new dataset $(X_{new}, y_{new})$, $X_{new} \in \real^{n \times p},~y_{new} \in \real^n$
where $X_{new} \neq X$. 
In this case, assuming we observe some new feature matrix $X_{new}$, and we are interested in the
out-of-sample prediction error,
\begin{equation}
\label{eq:out_of_sample}
\Err_{out} = \Earg{\norm{\mu(X_{new}) - X_{new}\bbeta(y)}^2} + n\sigma^2,
\end{equation}
where 
$$
\bbeta(y) = (X_M^T X_M)^{-1} X_M^T y, \quad \hM(y) = M, 
$$
where $\hM$ is the model selection procedure that depends on the data. Analogous to $\Err_{\alpha}$, we
define 
\begin{equation}
\label{eq:out_of_sample:alpha}
\Err_{out, \alpha} = \Earg{\norm{\mu(X_{new}) - X_{new}\bbeta^*(y)}^2} + n\sigma^2,
\end{equation}
where 
$$
\bbeta^*(y) = (X_M^T X_M)^{-1} X_M^T y, \quad \hM(y^*) = M. 
$$
We want to point out that we do not place any assumption on how the feature matrix is sampled. Specifically,
we do not need to assume $X_{new}$ is sampled from the same distribution as $X$. Rather, we condition on
the newly observed matrix $X_{new}$. This is a distinction from cross validation which assumes the rows
of the feature matrix $X$ are $i.i.d$ samples from some distribution. Such assumption may not be satisfied
in practice.

Then in the low dimensional setting where $p < n$, we are able to construct an unbiased estimator for $\Err_{out, \alpha}$.
\begin{lemma}
\label{lem:unbiased:out}
Suppose $X \in \real^{n \times p}$ and $\rank{X} = p$.
Then if we further assume a linear model where
$$
\mu(X) = X\beta^0,
$$
where $\beta^0$ is the underlying coefficients. Assuming the homoscedastic model in \eqref{eq:model}, we have
\begin{multline}
\widehat{\Err}_{out,\alpha} = \norm{H_0 y^- - H_{\hM(y^*)} y}^2 + 2\tr(H_{0}^T H_{\hM(y^*)})\sigma^2 + n\sigma^2\\
- 2\tr(H_0^T H_0)\left(1 + \frac{1}{\alpha}\right)\sigma^2
\end{multline}
is unbiased for $\Err_{out, \alpha}$, where 
$$
H_0 = X_{new}(X^T X)^{-1} X^T, \quad H_{\hM(y^*)} = X_{new, \hM(y^*)} (X_{\hM(y^*)}^T X_{\hM(y^*)})^{-1} X_{\hM(y^*)}^T. 
$$
\end{lemma}

The proof of the lemma is analogous to that of Theorem \ref{thm:unbiased} noticing that
$$
H_0 \mu(X) = X_{new}(X^T X)^{-1} X^T X\beta^0 = X_{new} \beta^0 = \mu(X_{new}).
$$ 

Lemma \ref{lem:unbiased:out} provides an unbiased estimator for $\Err_{out, \alpha}$ for $p < n$
and $\mu(X)$ being a linear function of $X$. To bound the difference $\Err_{out, \alpha} - \Err_{out}$,
we might need to assume conditions similar to those introduced at the beginning of Section \ref{sec:bias:var}.
In the case where $p < n$, we might still hope that the matrices $H_0$,
$H_{\hM}$ will be close to projection matrices, and almost satisfy Assumptions \ref{A:nonexpansive}.
Thus, intuitively, $\Err_{out,\alpha}$ and $\Err_{out}$ will be close and
the estimator $\widehat{\Err}_{out,\alpha}$ will be a good estimator of $\Err_{out}$ when $p < n$. 
In simulations, we see that in the low-dimensional setting, the performance of $\Err_{out, \alpha}$
is comparable to that of cross validation.
However in the high-dimensional setting where $n < p$, the estimation of out-of-sample errors
remains a very challenging problem that we do not seek to address in the scope of this work. 

\subsection{Search degrees of freedom}

There is a close relationship between (in-sample) prediction error and the degrees of freedom
of an estimator. In fact, with a consistent estimator for the prediction error $\Err$, we
get a consistent estimator for the degrees of freedom.

Under the framework of Stein's Unbiased Risk Estimator, for any estimation rule $\hmu$,
we have
\begin{equation}
\label{eq:stein:risk}
\Err = \Earg{\norm{y - \hmu(y)}^2} + 2\sum_{i=1}^n \Covarg{\hmu_i(y), y_i},
\end{equation}
where $\hmu_i$ is the $i$-th coordinate of $\hmu$. For almost differentiable $\hmu$'s, \cite{sure}
showed the covariance term is equal to 
\begin{equation}
\label{eq:stein}
\Covarg{\hmu_i(y), y_i} = \sigma^2 \Earg{\frac{\partial \hmu_i}{\partial y_i}}.
\end{equation}
The sum of the covariance terms, properly scaled is also called the degrees of freedom.
\begin{equation}
\label{eq:df}
\df = \sigma^{-2} \sum_{i=1}^n \Covarg{\hmu_i(y), y_i} = \sum_{i=1}^n \Earg{\frac{\partial \hmu_i}{\partial y_i}}.
\end{equation}
However, in many cases, the analytical forms of $\hmu$ are very hard to compute or there is
none. In such cases, the computation of its divergence is only feasible for very special $\hmu$'s 
\citep{zou2007degrees}. Moreover, for discontinuous $\hmu$'s which are under consideration
in this work, \cite{mikkelsen2016degrees} showed that there are further correction terms for
\eqref{eq:stein} to account for the discontinuities. In general, these correction terms do
not have analytical forms and are hard to compute. Intuitively, due to the search involved in constructing $\hmu = H_{\hM} y$,
it will have larger degrees of freedom than $\tr(H_{\hM})$ which treats the hat matrix as
fixed. We adopt the name used in \cite{tibshirani2014degrees} to call it ``search degrees of
freedom''.

We circumvent the difficulty in computing $\partial \hmu_i / \partial y_i$ by providing an
asymptotically unbiased estimator for $\Err$. Formally,
\begin{equation}
\label{eq:df:estimate}
\widehat{\df} = \frac{1}{\sigma^2}\left[\widehat{\Err}_{\alpha}^{(I)} - \norm{y - \hmu}_2^2\right],
\end{equation}
where $\widehat{\Err}_{\alpha}^{(I)}$ is defined as in \eqref{eq:marginalized}. Using the discussion
in Section \ref{sec:tradeoff}, we choose $\alpha = n^{-1/4}$. Notice that
such approach as above is not specific to any particular model search procedures involved in constructing $\hmu$. Thus it offers a unified approach
to compute degrees of freedom for any $\hmu = H_{\hM(y)} y$ satisfying the appropriate assumptions in
Section \ref{sec:bias:var}. We illustrate this flexibility by computing the search degrees of freedom
for the best subset selection where there has been no explicitly computable formula. 

Prediction error estimates may also be used for tuning parameters. For example, if the model
selection procedure $\hM$ is associated with some regularization parameter $\lambda$, we find
the optimal $\lambda$ that minimizes the prediction error of $\hmu_{\lambda}$
\begin{equation}
\label{eq:lambda}
\lambda_{optimal} = \min_{\lambda} \Earg{\norm{y_{new} - H_{\hM_{\lambda}(y)} \cdot y}^2_2},
\end{equation}
where the expectation is taken over both $y_{new}$ and $y$.
\cite{shen2002adaptive} shows that this model tuning criterion will yield an adaptively
optimal model which achieves the optimal prediction error as if the tuning parameter were
given in advance.


Using the relationship in \eqref{eq:stein:risk} and \eqref{eq:df}, we easily see that the $C_p$ type criterion \eqref{eq:lambda}
is equivalent to the AIC criterion using the definition of degrees of freedom \eqref{eq:df}. 
Analogously, we can also propose the BIC criterion as
$$
BIC = \frac{\norm{y - \hmu}_2^2}{n\sigma^2} + \frac{\log n}{n}\widehat{\df}
$$
\cite{aic_bic} points out that compared with the $C_p$ or AIC criterion, BIC tends to
recover the true underlying sparse model and recommends it if sparsity is the major concern.

\section{Simulations}
\label{sec:simulation}

In this work, we propose a method for risk estimation for a class of ``select and estimate''
estimators. One remarkable feature of our method is that it provides a consistent estimator
of the prediction error for a large class of selection procedures under general, mild conditions.
To demonstrate this strength, we provide simulations for two selection procedure under various
setups and datasets. The two estimators are the OLS estimator after best subset selection and
relaxed Lasso, which we denote as $\hmu_{best}$ and $\hmu_{relaxed}$. In particular,
$$
\hmu(y) = X_{\hM} (X_{\hM}^T X_{\hM})^{-1} X_{\hM}^T y,
$$ 
where $\hM$ is selected by the best subset selection and Lasso at a fixed $\lambda$ respectively
using the original data $y$. In their Lagrangian forms, best subset selection and Lasso at fixed
$\lambda$ can be written as
$$
\min_{\beta} \frac{1}{2} \norm{y - X\beta}_2^2 + \lambda \norm{\beta}_k
$$ 
where $k=0$ for best subset selection and $k=1$ for Lasso. Thus, by showing the good performances
(in simulation) of our estimator at both $k=0$ and $k=1$, we believe the
good performance would persist for all the non-convex optimization problems with $0 \leq k < 1$. 
In the simulation, we always marginalize over different randomizations to reduce variance. Specifically,
we use Algorithm \ref{alg:marginalized} to compute $\widehat{\Err}_{\alpha}^{(I)}$ which we use in
all of the comparisons below.

In the following simulations, we compare both the bias and variances of our estimator 
$\widehat{\Err}_{\alpha}^{(I)}$ with the $C_p$ estimator, cross validation as well as the parametric bootstrap
method proposed in \cite{efron_cp}. In particular, to ensure
fairness of comparison, we use Leave-One-Out cross validation in all of our simulations. Most
of our simulations are for in-sample prediction errors with some exceptions of comparing
the out-of-sample estimator $\widehat{\Err}_{out,\alpha}^{(I)}$ in Section \ref{sec:out_of_sample} to cross validation for estimating out-of-sample
prediction errors. 
To establish a ``known'' truth to compare to, we use mostly synthetic data, with some of the synthetic datasets
generated from a diabetes dataset. In the following simulations, we call our estimator ``additive''
due to the additive randomization used in the estimation. Cross validation is abbreviated as ``CV''.
The true prediction error is evaluated through Monte-Carlo sampling since we have access to the ``true''
underlying distribution. 
We assume the variance $\sigma^2$ is unknown and estimate it with the OLS residuals when $p < n$.
In the high-dimensional setting, we use the methods in \cite{reid2013study} to estimate $\sigma^2$. 

\subsection{Relaxed Lasso estimator}

We perform simulation studies for the prediction error and degrees of freedom estimation for the
relaxed Lasso estimator. Unless stated otherwise, the target of prediction error estimation is the
in-sample prediction error:
$$
\Err = \Earg{\norm{y_{new} - \hmu_{relaxed}(y)}_2^2}, \quad y_{new} \sim N(\mu(X), \sigma^2 I) \perp y.
$$
According to the framework of SURE \cite{sure}, the degrees of freedom of the estimator $\hmu_{relaxed}$ can be
defined as
$$
\df = \sum_{i=1}^n \frac{\Covarg{\hmu_{relaxed, i}, y_i}}{\sigma^2}
$$
which is the target of our estimation. We first study the performance of the prediction error estimation.

\subsubsection{Prediction error estimation}
\label{sec:in:sample}

In the following, we describe our data generating distribution as well as the parameters used in the simulation.
\begin{itemize}
\item The feature matrix $X \in \real^{n \times p}$ is simulated from an equi-correlated covariance matrix
with normal entries. The correlation is $\rho=0.3$. 
\item $y$ is generated from a sparse linear model,
$$
y = X\beta^0 + \epsilon, \quad \epsilon \sim N(0, \sigma^2 I),
$$ 
where
$$
\beta^0 = (\underbrace{snr,\dots, snr}_{s},0,\dots,0)
$$
and $snr$ is the signal-to-noise ratio and $s$ is the sparsity of $\beta^0$. 
\item We fit a Lasso problem with $\lambda = \kappa \lambda_0$, where
$$
\lambda_{min} = \Earg{\norm{X^T \epsilon^{'}}_{\infty}}, \quad \epsilon^{'} \sim N(0, \sigma^2 I),
$$ 
is the level where noise below which noise starts to enter the Lasso path \cite{negahban2009unified}
and we choose $\kappa > 1$.
\item The parameter $\alpha$ as defined in \eqref{eq:random:additive} is taken to be approximately $n^{-\frac{1}{4}}$. 
\end{itemize}

We compare the performances of the estimators for different settings. We take $n = 100$, and $p=50, 200, 400$
and sparsity to be $s = 10, 20$. Since $n^{-\frac{1}{2}} = 10$, $s = 20$ is the more dense signal situation.
We take $\kappa$ to be $1.1$ for the low-dimensional setting and $1.5$ for the high-dimensional setting. 
The randomization parameter $\alpha = 0.25 \approx n^{-1/4}$. We see
in Figure \ref{fig:barplot} that in all settings $\widehat{\Err}_{\alpha}^{(I)}$ provides an unbiased estimator that
has small variance. Remarkably, notice that the variance of our estimator is comparable to 
the dotted the black lines are the standard error of the true prediction error estimated from Monte-Carlo sampling,
which is probably the best one can hope for. $\widehat{\Err}_{\alpha}^{(I)}$ clearly outperforms both $C_p$ and cross validation.
Its performance is comparable to the parametric bootstrap estimator
in the sparse scenario although parametric bootstrap seems to have more extreme values. 
Our estimator also performs slightly better in the more dense scenario $s=20$ in panel $3$ of Figure \ref{fig:barplot}.
In the dense signal situation, the model selected by Lasso is often misspecified. We suspect that in this situation, 
that parametric bootstrap overfits the data in this situation, causing a slight bias downwards. 
The $C_p$ estimator is always biased down because it does not take into account
the ``degrees of freedom'' used for model search. On the other hand, cross validation has an upward bias for in-sample
prediction error. However, this bias is two fold. First, the extra randomness in the new feature matrix will cause
the out-of-sample prediction error to be higher. However, comparing panel $3$ and $4$ of Figure \ref{fig:barplot},
we see that when the signal is more dense $s=20$ in panel $3$, cross validation has a much larger bias than when the dimension is
higher $p=400$ in panel $4$. This suggests that cross validation might be susceptible to model misspecifications as well.
With less sparse signals, the model selected by Lasso is not stable or consistent, causing
cross validation to behave wildly even when we only leave out one observation at a time. In contrast, in all of the
four settings, our estimator $\widehat{\Err}_{\alpha}^{(I)}$ provides an unbiased estimator with small variance. 

\begin{figure}
\begin{center}
\includegraphics[width=.9\textwidth]{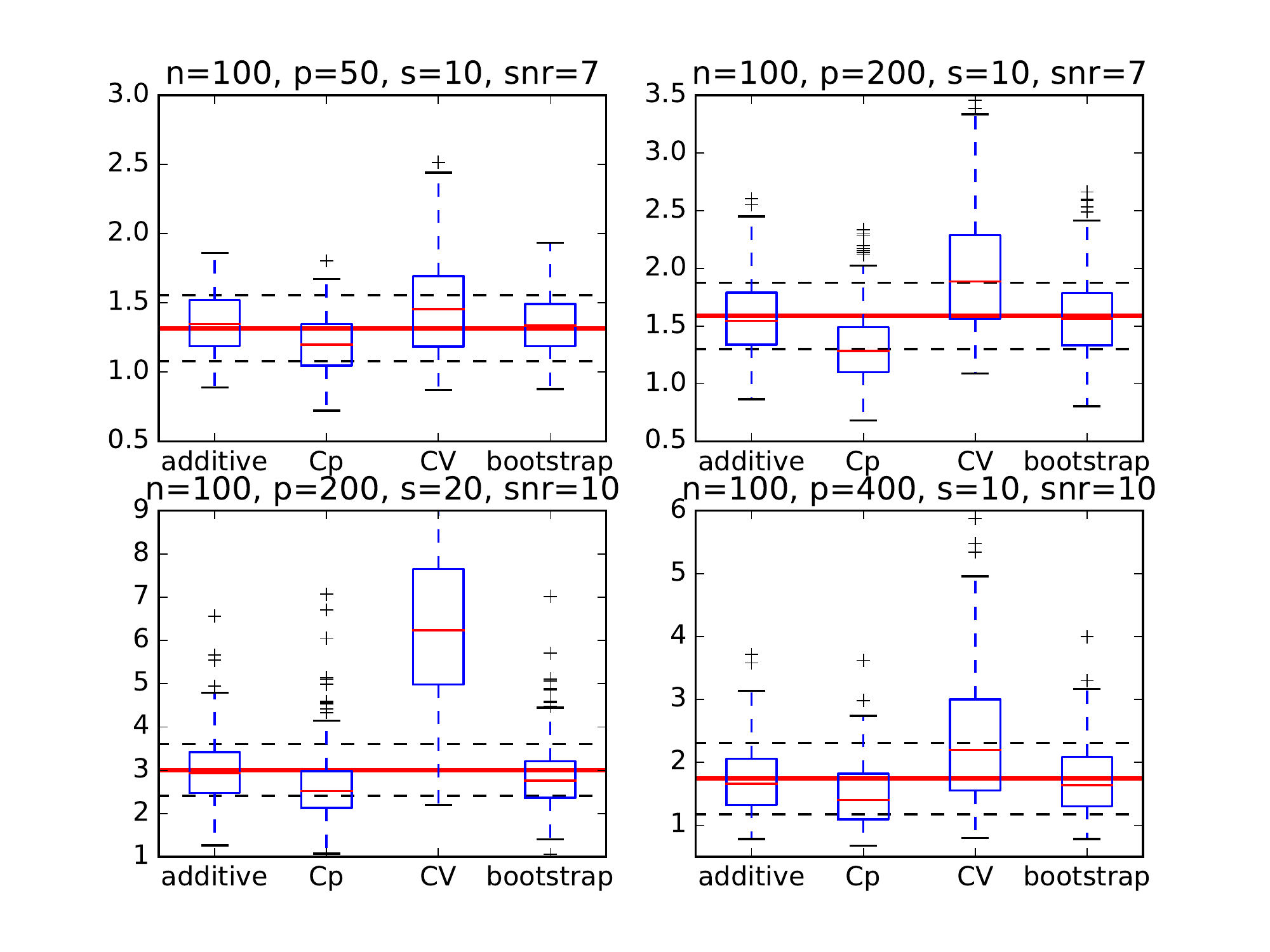}
\end{center}
\caption{Comparison of different estimators for different $n,p,s,snr$. The red horizontal line is the true prediction
error estimated by Monte-Carlo simulation with the dashed black lines denoting its standard deviation.}
\label{fig:barplot}
\end{figure}

This phenomenon persists when we vary the penalty parameter $\lambda$. For a grid of $\lambda$'s with varying
$\kappa$'s from $[0.2, 1.6]$, we see from Figure \ref{fig:lambda} that cross validation error is always
overestimates the in-sample prediction error. Moreover, the amount of over estimation highly depends on the data
generating distribution. In both panels of Figure \ref{fig:lambda}, $n=100,~p=200,~snr=7.$, and the only
difference is the sparsity is $s=10$ for Figure \ref{fig:more_sparse} and $s=20$ for Figure \ref{fig:less_sparse}.
Using the same dimensions for $X$, we seek to control the extra randomness by using a different $X$ for the
validation set. However,
the change in the sparsity level alone has huge impact for the cross validation estimates of the prediction error.
The curve by cross validation is also more kinky due to its bigger variance.
However, in both scenarios, $\widehat{\Err}_{\alpha}^{(I)}$ hugs the true prediction error. 

\begin{figure}
    \centering
    \begin{subfigure}[b]{0.4\textwidth}
        \includegraphics[width=\textwidth]{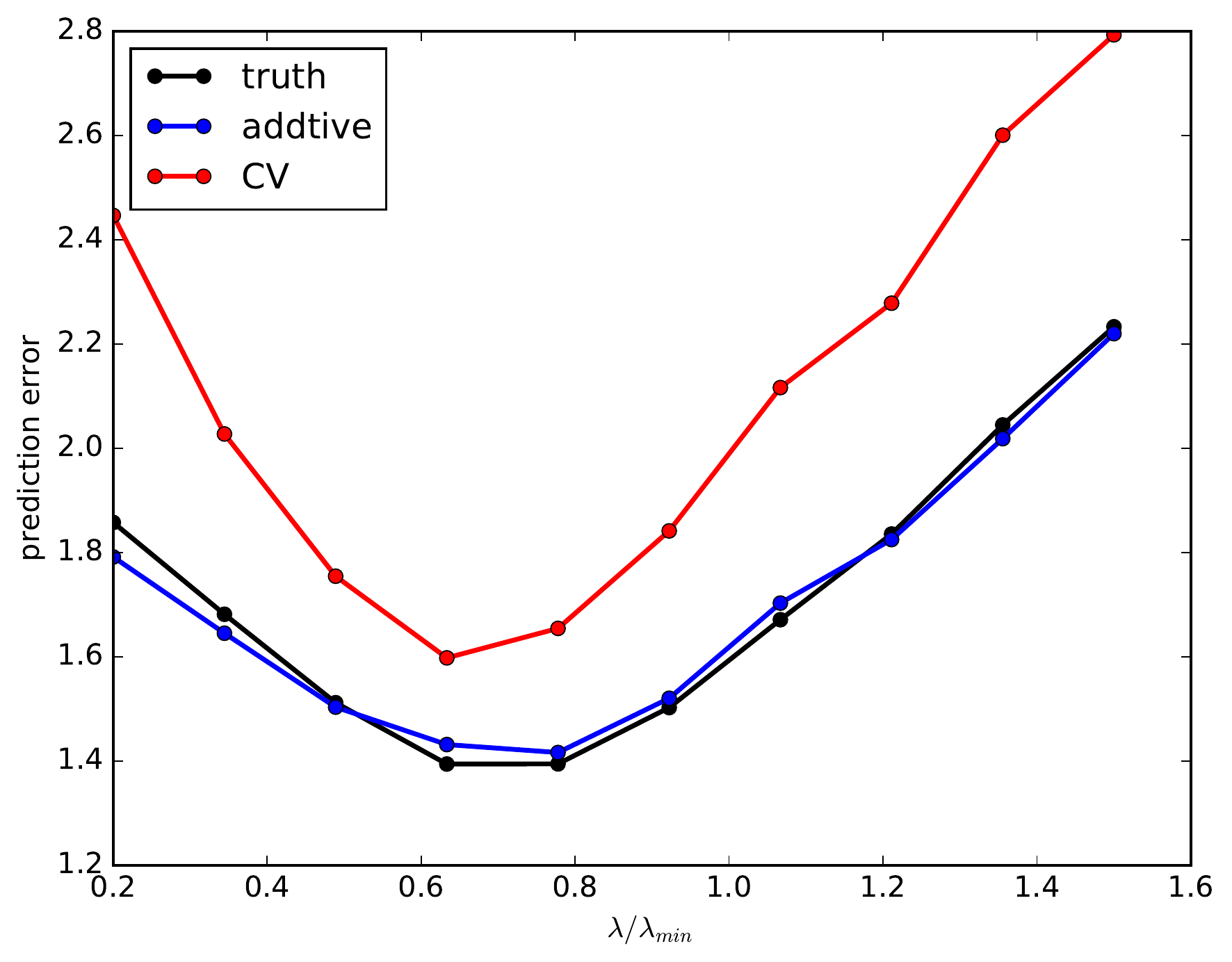}
        \caption{$n=100,~p=200,~s=10$}
        \label{fig:more_sparse}
    \end{subfigure}
    ~     \begin{subfigure}[b]{0.4\textwidth}
        \includegraphics[width=\textwidth]{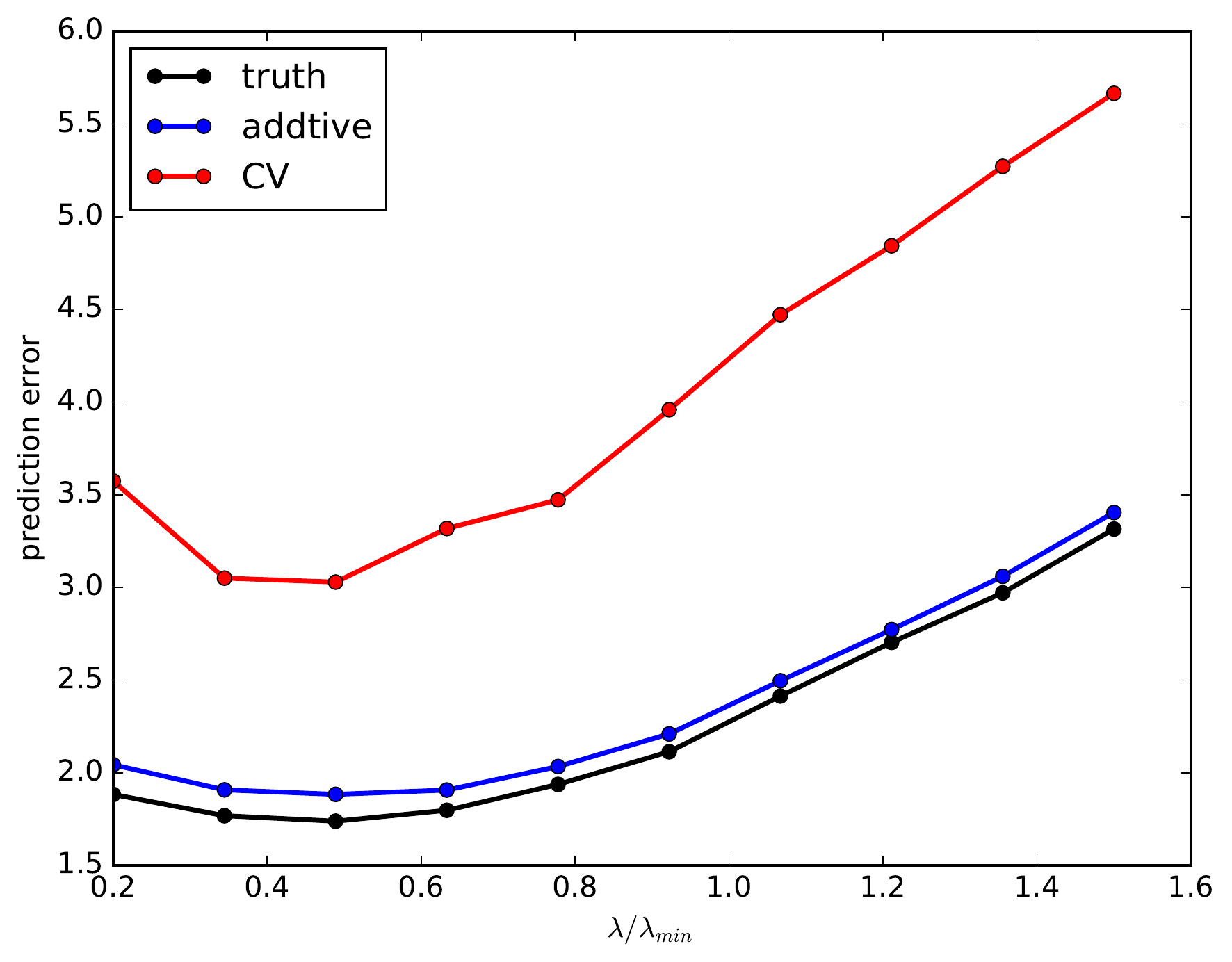}
        \caption{$n=100,~p=200,~s=20$}
        \label{fig:less_sparse}
    \end{subfigure}
    \caption{Estimation of prediction errors for different $\lambda$'s. Cross validation
is always biased upwards. However, the bias depends on the data generating distribution.}\label{fig:lambda}
\end{figure} 

\subsubsection{Degrees of freedom}

In this section, we carry out a simulation study for our estimate of the degrees of freedom of the relaxed
Lasso estimator $\hmu_{relaxed}$. We take the 64 predictors in the diabetes dataset \citep{lars} to be our
feature matrix $X$, which
include the interaction terms of the original ten predictors. The positive cone condition is violated on
the 64 predictors \citep{lars, zou2007degrees}. We use the response vectors $y$ to compute the OLS estimator
$\hbeta_{ols}$ and $\hat{\sigma}_{ols}$ and then synthetic data is generated through
$$
y = X \hbeta_{ols} + \epsilon, \quad \epsilon \sim N(0, \sigma_{ols}^2 I).
$$

We choose $\lambda$'s to have different ratios $\kappa \in \{0.05, 0.1, 0.15, 0.2, 0.25\}$. Figure
\ref{fig:df} shows the estimates of degrees of freedoms by our method as in \eqref{eq:df:estimate}
and the naive estimate $\hat{\df}_{naive} = |\hM|$
compared with the truth computed by Monte-Carlo sampling. The naive $C_p$ estimator always underestimate
the degrees of freedom, not taking into account the inflation in degrees of freedom after model search.
However, our estimator as defined in \eqref{eq:df} provides an unbiased estimation for the true degrees of
freedom for the relaxed Lasso estimator $\hmu_{relaxed}$. 

\begin{figure}
\begin{center}
\includegraphics[width=.8\textwidth]{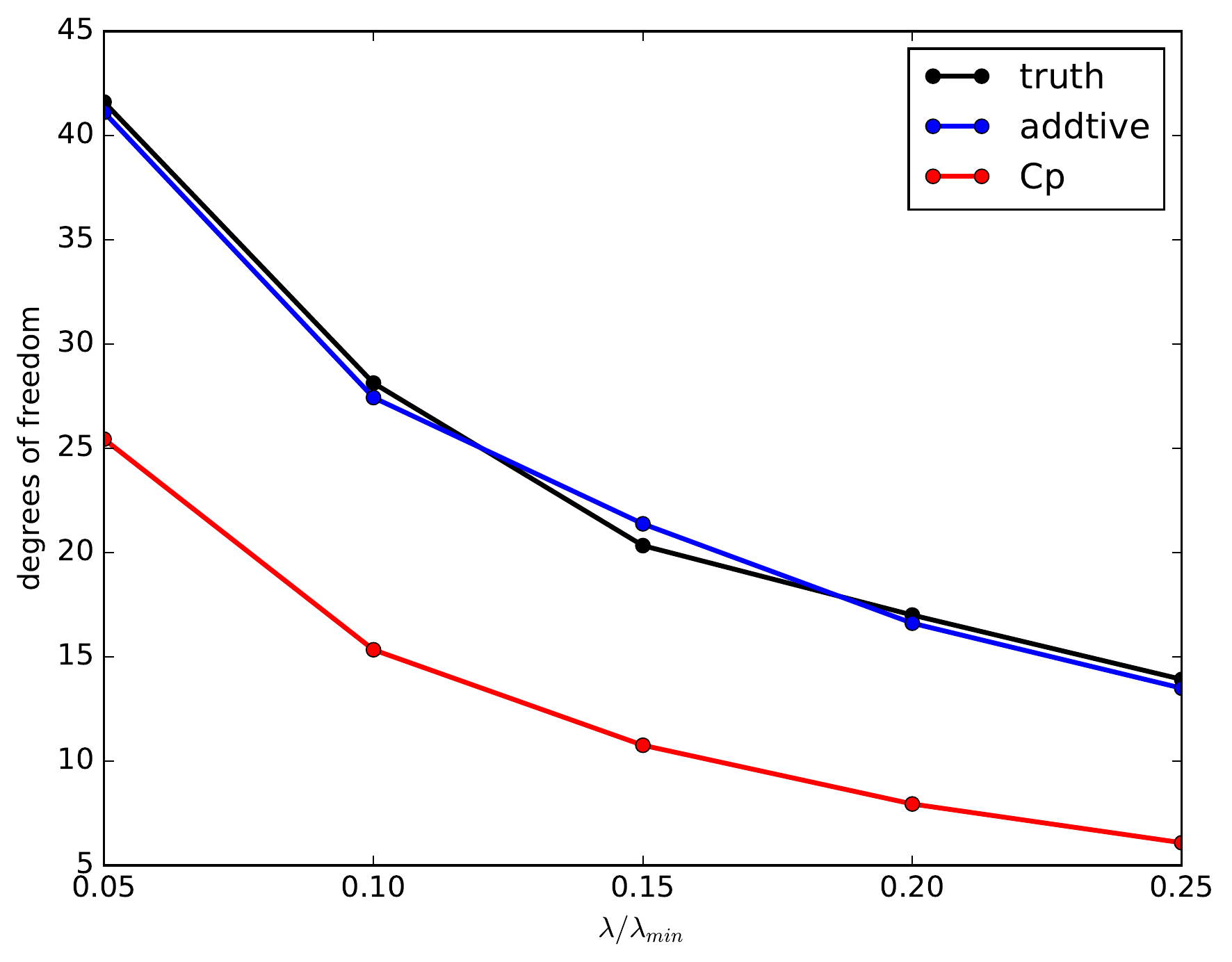}
\end{center}
\caption{Comparison of estimates of degrees of freedom by cross validation, $\hat{\df}$ in \eqref{eq:df:estimate}
and $\hat{\df}_{naive} = \abv{\hM}$ at different $\lambda$'s. $\alpha = 0.25 \approx n^{-1/4}$.}
\label{fig:df}
\end{figure}

\subsubsection{Out-of-sample prediction errors}

Finally, we test the unbiasedness of the proposed estimator in Section \ref{sec:out_of_sample} for out-of-sample
prediction error. We compare with cross validation in the low-dimensional setting where $p=20$ and $p=50$ respectively.
In this section only, our target is the out-of-sample prediction error
$$
\Err_{out} = \Earg{\norm{X_{new}\beta^0 - X_{new, \hM(y)}\bbeta(y)}^2} + n\sigma^2,
$$
where $\bbeta$ is the relaxed Lasso estimator and $\hM(y)$ is the nonzero set of the Lasso solution at $\lambda$.
We still abbreviate our estimator as ``additive'' and compare with the out-of-sample prediction error by
cross validation.

We see in Figure \ref{fig:out} that the estimator proposed in Section \ref{sec:out_of_sample} is roughly unbiased for
out-of-sample prediction error. Its performance is comparable with cross validation in both settings, with a slightly
larger variance. However, as pointed in Section \ref{sec:out_of_sample}, our estimator does not assume any assumptions
on the underlying distribution of the feature matrix $X$. 

\begin{figure}
\begin{center}
\includegraphics[width=.9\textwidth, height=0.6\textwidth]{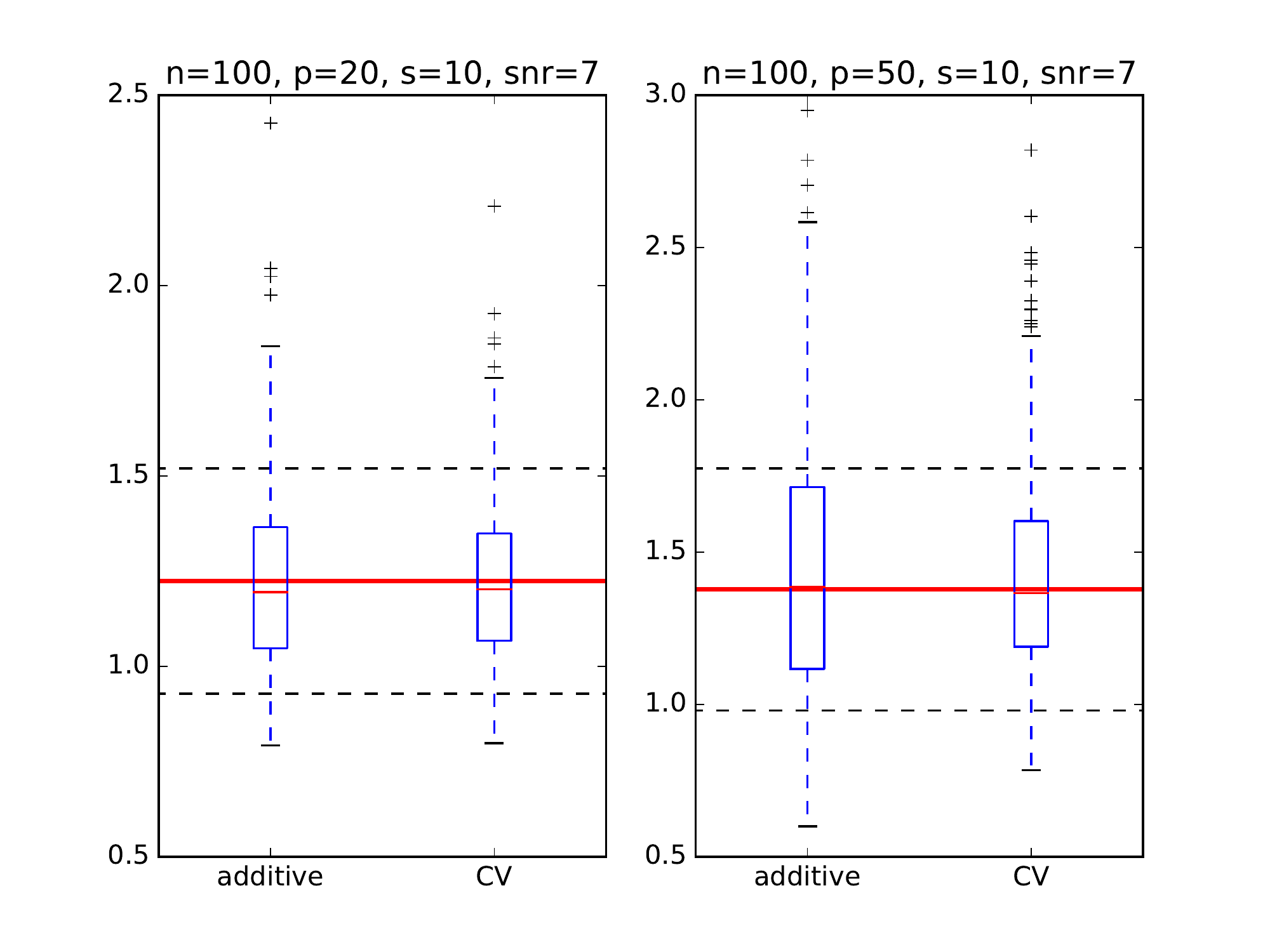}
\end{center}
\caption{Out of sample prediction error by $\widehat{\Err}_{out, \alpha}^{(I)}$ and cross validation respectively.}
\label{fig:out}
\end{figure}

\subsection{Best subset selection}

The $C_p$ estimator was originally proposed for picking the model size in best subset selection. One aspect
that often gets neglected is that for any $k < p$, where $p$ is the number of features to choose from,
there are more than one models of size $k$ to choose from. And the best subset of size $k$ already includes a selection
procedure that needs to be adjusted for. To illustrate this problem, we generate a feature matrix $X$ of dimension
$100 \times 6$ with i.i.d standard normal entries. And $y$ is generated from a linear model of $X$
$$
y = X\beta + N(0, 1), \quad \beta = (1,2,3,4,5,6).
$$

For each subset of size $k=1,\dots,6$, we estimate the prediction error of the best subset of size $k$ using both
$C_p$ and $\widehat{\Err}_{\alpha}^{(I)}$. The true prediction error is evaluated using Monte-Carlo sampling. From
Figure \ref{fig:best_subset}, we see that $C_p$ is indeed an under estimate for the prediction error for best
subset selection. The bias is bigger when $k = 2,3,4$ when there are more potential submodels to select from. 
In contrast, $\widehat{\Err}_{\alpha}^{(I)}$ hugs the true prediction error at every subset size $k$.

\begin{figure}
\begin{center}
\includegraphics[width=.9\textwidth]{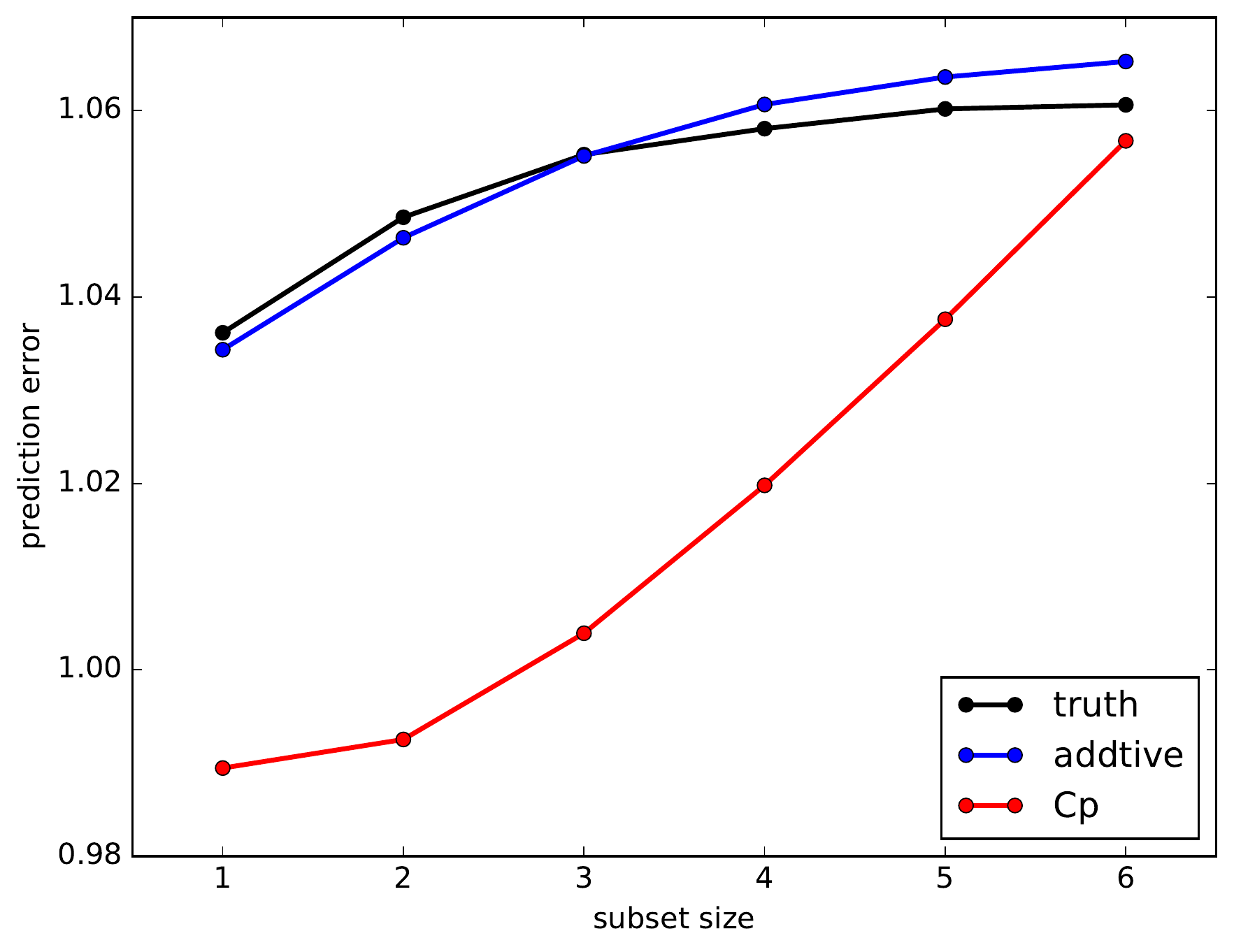}
\end{center}
\caption{Comparison of different estimates of prediction errors}
\label{fig:best_subset}
\end{figure}

\section{Discussion}
\label{sec:discussion}

In this work, we propose a method for estimating the prediction error after some data snooping in
selecting a model. Remarkably, our estimation is not specific to any particular model selection
procedures so long as it does not select too many variables to include in the model and it picks up
some signals in the data. Different examples are considered. 

In the following, we propose two more aspects of the problem that deserve attention but we do not
seek to address in this work.
\begin{itemize}
\item We mainly focus on ``in-sample'' prediction errors, with the exception of Section \ref{sec:out_of_sample}.
But as pointed in Section \ref{sec:out_of_sample}, although we can provide a consistent estimator of the
(in-sample) prediction error in high dimensions, the same is not true for out-of-sample errors. \cite{cv_high_dim} 
points out that the same difficulty exists for cross validation as well. Under what assumptions can we provide
a good estimator for out-of-sample prediction error in high dimensions remains a very interesting question.
\item Throughout the work, we assume that the data comes from a homoscedastic normal model \eqref{eq:model}.
Some simulations show that the performance of our estimator persists when the noise in the data is subGaussian.
The authors of \cite{randomized_response} pointed out that it is important that the tail of the randomization
noise is heavier than that of the data. Since we add Gaussian noise for randomization, we suspect that
the normal assumption on the data can be replaced by a subGaussian assumption. Alternatively, we may investigate
what other randomization noise we may add to the data when we have heavier-tailed data. 
\end{itemize}

\section{Proof of the lemmas}
\label{sec:proof}

The following lemmas are essential in proving the main theorems and lemmas which we introduce
first. 

\begin{lemma}  
\label{lem:max:4th}
Suppose $Z_i \sim N(0, 1)$ but not necessarily independently distributed. Let $W_n = \max_{1 \leq i \leq n} Z_i^4$,
then 
\begin{equation}
\label{eq:max:moment}
\Earg{W_n} \leq 18 (\log n)^2.
\end{equation}
\end{lemma}

\begin{proof}
For any integer $k > 0$, we have
$$
\begin{aligned}
\left[\Earg{W_n}\right]^k &\leq \Earg{W_n^k} = \Earg{\max_{i=1,\dots,n} Z_i^{4k}}\\
&\leq \sum_{i=1}^n \Earg{Z_i^{4k}} = n\left[\Earg{Z_i^{4k}}\right].
\end{aligned}
$$
Thus we have for any positive integer $k > 0$,
\begin{equation}
\label{eq:any:k}
\begin{aligned}
\Earg{W_n} &\leq \left[n\Earg{Z_i^{4k}}\right]^{\frac{1}{k}}\\
&= \left[n 2^{\frac{4k}{2}} \Gamma\left(\frac{4k+1}{2}\right) / \sqrt{\pi}\right]^{\frac{1}{k}}\\
\text{(Sterling's formula)} &\leq \left[n 2^{\frac{4k}{2}} \sqrt{\frac{2.1\pi}{(4k+1)/2}} 
\left(\frac{(4k+1)/2}{e}\right)^{\frac{4k+1}{2}}/\sqrt{\pi} \right] ^{\frac{1}{k}}\\
& = 4n^{\frac{1}{k}} \left[\sqrt{2.1} \left(\frac{4k+1}{2}\right)^{2k} \exp\left(-\frac{4k+1}{2}\right)\right]^{\frac{1}{k}}\\
&= 4n^{\frac{1}{k}} (2.1)^{\frac{1}{2k}} \left(\frac{4k+1}{2}\right)^2 \exp\left(-2-\frac{1}{2k}\right)\\
&\leq \frac{4}{\mathrm{e}^2} n^{\frac{1}{k}} (2.1)^{\frac{1}{2k}} \left(\frac{4k+1}{2}\right)^2.
\end{aligned}
\end{equation}
We choose $k$ to minimize the bound on the right hand side. Let
$$
f(k) = \log\left[n^{\frac{1}{k}} (2.1)^{\frac{1}{2k}} \left(\frac{4k+1}{2}\right)^2\right]
= \frac{1}{k} \log n + \frac{1}{2k} \log(2.1) + 2 \log \left(\frac{4k+1}{2}\right)
$$
Take derivatives with respect to $k$, we have
$$
f^{'}(k) = -\frac{1}{k^2} \log n - \frac{1}{2k^2} \log(2.1) + \frac{8}{4k+1},
$$
and it is easy to see the maximum of $f(k)$ is attained at $k = \frac{1}{2}\log(\sqrt{2.1} n)$.
And the minimum of $f(k)$ is
$$
f\left(\frac{\log(\sqrt{2.1} n)}{2}\right) = 2 + \frac{\log(2.1)}{\log(\sqrt{2.1}n)} + 2\log \left(2\log(\sqrt{2.1}n) + 1/2\right).
$$
Since \eqref{eq:any:k} holds for any integer $k > 0$, it is easy to see
$$
\Earg{W_n} \leq 4 \times 4.5(\log n)^2 \leq 18(\log n)^2.
$$
\end{proof}

Based on Lemma \ref{lem:max:4th}, we prove Lemma \ref{lem:max:rand},

\begin{proof}
First since $H_{\hM}$ can only take values in
$$
\left\{H_{M_i},~~i=1, \dots, \cardM \right\},
$$
we have
\begin{equation}
\label{eq:norm:bound}
\norm{H_{\hM} Z}_2^2 \leq \max_{i=1,\dots,\cardM} \norm{H_{M_i} Z}_2^2.
\end{equation}
We use the short hand $H_i$ to denote $H_{M_i}$. Since for any $i=1,\dots,\cardM$,
$H_i$ has eigen decomposition,
$$
H_i = V_i D V_i^T, \quad D = \diag(d_1, \dots, d_K), \quad V_i^T V_i = I_{K \times K}
$$
where $0 \leq d_i \leq 1$ and $\rank{H_i} \leq K$. Thus it is easy to see
$$
\norm{H_i Z}_2^2 = \norm{D V_i^T Z}_2^2 \leq \sum_{j=1}^K \xi_{ij}^2, \quad \xi_{ij} \sim N(0, \sigma^2). 
$$
Note here that we do not assume any independence structure between $\xi_{ij}$'s. In fact, they are most likely not
independent.

Combining with \eqref{eq:norm:bound}, we have
$$
\left(\norm{H_{\hM} Z}_2^2\right) \leq K \max_{\substack{i=1,\dots, \cardM, \\ j=1,\dots,K}} \xi_{ij}^2
$$
Therefore,
$$
\left(\norm{H_{\hM} Z}_2^2\right)^2 \leq K^2\max_{\substack{i=1,\dots,\cardM, \\ j=1,\dots, K}} \xi_{ij}^4.
$$
Thus, using Lemma \ref{lem:max:4th}, it is easy to get the conclusion of this lemma.
\end{proof}

Finally, we prove Lemma \ref{lem:linear:alpha} as follows.
\begin{proof}
First notice that for hat matrix of form in \eqref{eq:hat:matrix}, we have
$$
\Earg{\norm{\hmu(y) - \mu}^2} = \sum_{i=1}^{\cardM} \int_{U_i} \norm{H_iy - \mu}^2 \phi(y; \mu, \sigma^2) dy,
$$
where $H_i$ is short for $H_{M_i}$ and $\phi(\cdot; \mu, \sigma^2)$ is the density for $N(\mu, \sigma^2 I)$.
Let $\sigma^2 \leq \tau \leq (1+\delta)\sigma^2$, where $\delta$ is defined in Assumption \ref{A:accuracy},
and we define,
$$
g(\tau) = \Earg{\norm{\hmu(u) - \mu}^2}, \quad u \sim N(\mu, \tau I).
$$
We note that $g$ is differentiable with respect to $\tau$ and
\begin{multline}
\frac{1}{n}\abv{\Earg{\norm{\hmu(y+\omega) - \mu}^2} - \Earg{\norm{\hmu(y) - \mu}^2}}
= \frac{1}{n} \abv{g((1+\alpha)\sigma^2) - g(\sigma^2)},\\
0 < \alpha \leq \delta.
\end{multline}
Moreover, we have
\begin{equation}
\label{eq:derivatives}
\begin{aligned}
\partial_{\tau}g(\tau) &= \sum_{i=1}^{\cardM} \int_{U_i} \norm{H_i u - \mu}^2 \partial_{\tau}\left[\frac{1}{(\sqrt{2\pi\tau})^n}\exp\left(-\frac{\norm{u-\mu}_2^2}{2\tau}\right)\right] du\\
&=\frac{1}{2\tau} \sum_{i=1}^{\cardM} \int_{U_i} \norm{H_i u - \mu}^2 \left[\frac{\norm{u-\mu}^2}{\tau} - n\right] \phi(u; \mu, \tau) du\\
&= \frac{1}{2\tau} \Earg{\norm{H_{\hM(u)} u - \mu}^2 \left[\frac{\norm{u-\mu}^2}{\tau} - n\right]}\\
|\partial_{\tau}g(\tau)| &\leq \frac{1}{2\tau} \Earg{\left[\norm{H_{\hM(u)} u - \mu}^2\right]^2}^{\frac{1}{2}}
\Earg{\left[\frac{\norm{u-\mu}^2}{\tau} - n\right]^2}^{\frac{1}{2}}
\end{aligned}
\end{equation}
Using Assumptions \ref{A:trace:bound} and \ref{A:accuracy}, we assume there is a universal constant $c_1> 0$ such that
$$
\begin{aligned}
K^2\log p &\leq c_1 \sqrt{n}, \\
\Earg{\left[\norm{(I-H_{\hM(u)})\mu}^2\right]^2} &\leq c_1^2 n, \quad u\sim N(\mu, \tau I), ~\forall~\tau \in [\sigma^2, (1+\delta)\sigma^2]. 
\end{aligned}
$$
Thus assuming $(I-H_{\hM})H_{\hM} = 0$, we have
$$
\begin{aligned}
\Earg{\left[\norm{H_{\hM(u)} u - \mu}^2\right]^2} &= \Earg{\left[\norm{(I-H_{\hM(u)})\mu - H_{\hM(u)}\epsilon'}^2\right]^2}, \quad \epsilon' \sim N(0, \tau I)\\
&\leq \Earg{2\left(\norm{(I-H_{\hM(u)})\mu}^2\right)^2 + 2\left(\norm{H_{\hM} \epsilon'}^2\right)^2}\\
&\leq 2c_1^2 n + 36c_1^2 n \tau^2 \\
&\leq 38c_1^2 n \max(1,\tau^2)
\end{aligned}
$$
where the last inequality uses Lemma \ref{lem:max:rand} and the fact $\cardM \leq p^K$. Moreover, note that
$\norm{u-\mu}^2/\tau$ is a $\chi^2_{n}$ distribution with mean $n$, thus
$$
\Earg{\left[\frac{\norm{u-\mu}^2}{\tau} - n\right]^2} = \Vararg{\chi^2_n} = 2n.
$$
Combining the above inequalities with \eqref{eq:derivatives} and we have
$$
\abv{\partial_{\tau} g(\tau)} \leq \frac{9}{2}c_1\max\left(\frac{1}{\tau},1\right) n, \quad \tau \in [\sigma^2, (1+\delta)\sigma^2]. 
$$
Therefore, for any $0 < \alpha \leq \delta$, we have
$$
\begin{aligned}
&\frac{1}{n}\abv{\Earg{\norm{\hmu(y+\omega) - \mu}^2} - \Earg{\norm{\hmu(y) - \mu}^2}}\\
=& \frac{1}{n} \abv{g((1+\alpha)\sigma^2) - g(\sigma^2)} \\
\leq& \frac{C}{\sigma^2} \cdot \alpha\sigma^2 \leq C\cdot \alpha, 
\end{aligned}
$$
where we take $C = \frac{9}{2}c_1\max(\sigma^2,1)$ is a universal constant assuming fixed $\sigma^2$.
\end{proof}

Now we prove Lemma \ref{lem:norm:var}.
\begin{proof}
Using the singular value decomposition of $A$, it is easy to see we can reduce the
problem to the case where $A$ is a diagonal matrix. Thus, without loss of generality,
we assume
$$
A = \diag(a_1, a_2, \dots, a_n).
$$
Then we see that
$$
\begin{aligned}
\Earg{\norm{AZ}_2^2} &= \Earg{\left(\sum_{i=1}^n a_i^2 Z_i^2\right)} = \left(\sum_{i=1}^n a_i^2\right) , \\
\Earg{\norm{AZ}_2^2}^2 &= \Earg{\left(\sum_{i=1}^n a_i^2 Z_i^2\right)}^2 = \left(3\sum_{i=1}^n a_i^4 + \sum_{i\neq j} a_i^2 a_j^2\right).
\end{aligned}
$$
Thus we deduce
$$
\Vararg{\norm{AZ}_2^2} = \left(2\sum_{i=1}^n a_i^4\right) = 2\tr\left(A^4\right).
$$ 
\end{proof}

{\bf Acknowledgement} The author wants to thank Professor Jonathan Taylor, Professor Robert Tibshirani,
Frederik Mikkelsen and Professor Ryan Tibshirani for useful discussions during this project.

\bibliographystyle{agsm}
\bibliography{paper}

\end{document}